\documentclass[11pt]{article}

\usepackage{amsmath}
\usepackage{amssymb}
\usepackage{amsthm}
\usepackage{graphics}
\usepackage[latin1]{inputenc}
\usepackage[english]{babel}
\usepackage[T1]{fontenc}
\usepackage{enumerate}
\usepackage{color}
\usepackage{hyperref}

\newtheorem{thm}{Theorem}[section]
\newtheorem{prop}{Proposition}[section]
\newtheorem{coro}{Corollary}[section]
\newtheorem{lemma}{Lemma}[section]
\newtheorem{rem}{Remark}[section]
\newtheorem{rems}{Remarks}[section]
\newtheorem{defi}{Definition}[section]
\newtheorem{hypo}{Hypothesis}[section]

\newcommand{\R}{\mathbb{R}}             
\newcommand{\N}{\mathbb{N}}             

\newcommand{\e}{\epsilon}

\newcommand{\Dn}{\Lambda_{g} }
\newcommand{\Dno}{\Lambda_{g_{0}} }
\newcommand{\Dngammaone}{\Lambda_{g,\Gamma_1,\Gamma_1} }
\newcommand{\Dngammaonetilde}{\Lambda_{\tilde{g},\Gamma_1,\Gamma_1} }

\newcommand{\Dngammaitilde}{\Lambda_{\tilde{g},\Gamma_i,\Gamma_i} }

\newcommand{\Dngammaij}{\Lambda_{g,\Gamma_{i},\Gamma_{j}} }
\newcommand{\Dngammai}{\Lambda_{g,\Gamma_{i},\Gamma_{i}} }

\newcommand{\Dngammaqzero}{\Lambda_{g_{0},q,\Gamma_{0},\Gamma_{0}} }
\newcommand{\Dngammaqone}{\Lambda_{g_{0},q,\Gamma_{1},\Gamma_{1}} }
\newcommand{\Dngammaqzerotoone}{\Lambda_{g_{0},q,\Gamma_{0},\Gamma_{1}} }
\newcommand{\Dngammaqij}{\Lambda_{g_{0},q,\Gamma_{i},\Gamma_{j}} }

\DeclareMathOperator{\Span}{span}

\newcommand{\ds}{\displaystyle}

\newcommand{\Section}[1]{\section{#1} \setcounter{equation}{0}}

\begin{document}
\title{Stability in the anisotropic Calder\'{o}n problem for Painlev\'e-Liouville Riemannian manifolds}
\author{Thierry Daud\'e \footnote{Research supported by the project ScattHoloGR (ANR-25-CE40-4883) funded by the French National Research Agency (ANR)} $^{\,1}$, Niky Kamran \footnote{Research supported by NSERC grant RGPIN 105490-2025} $^{\,2}$ and Fran{\c{c}}ois Nicoleau \footnote{Research supported by the French GDR Dynqua} $^{\,3}$\\[12pt]
$^1$  \small Universit\'e Marie et Louis Pasteur, CNRS, LmB (UMR 6623), F-25000, Besan{\c{c}}on, France. \\
\small CNRS - Universit\'e de Montr\'eal CRM - CNRS \\
\small  Email address: thierry.daude@univ-fcomte.fr, \\
$^2$ \small Department of Mathematics and Statistics, McGill University,\\ \small  Montreal, QC, H3A 0B9, Canada. \\
\small Email: niky.kamran@mcgill.ca \\
$^3$  \small  Laboratoire de Math\'ematiques Jean Leray, UMR CNRS 6629, \\ \small 2 Rue de la Houssini\`ere BP 92208, F-44322 Nantes Cedex 03. \\
\small Email: francois.nicoleau@univ-nantes.fr }





\maketitle


\begin{abstract} 
We study the question of stability of the global and partial anisotropic Calder\'on inverse problems for the class of Painlev\'e-Liouville Riemannian manifolds, that is compact $n$-dimensional manifolds with boundary $(M,g)$, where $M=[0,1]\times K\,$, $K$ is any smooth closed connected orientable manifold of dimension $n-1$ endowed with a Riemannian metric $g_K$, and $g=\alpha^4 g_{0}$ is any conformal deformation of the product metric $g_{0}=dx^2+g_{K}$ on $M$ which is compatible with the Painlev\'e block-separability of the Laplace-Beltrami operator $\Delta_{g_0}$. Given a pair of Painlev\'e-Liouville Riemannian manifolds $(M,g)$ and $(M,\tilde{g})$ satisfying the technical hypothesis \ref{Hyp}, denoting the corresponding Dirichlet-to-Neumann maps by $\Lambda_{g}$ and $\Lambda_{\tilde{g}}$, and assuming that $\lVert \Lambda_{g}-\Lambda_{\tilde{g}}\rVert_{\mathcal{B}(H^{1/2}(\partial M), H^{-1/2}(\partial M))}\ = \epsilon $, we show a logarithmic stability result for the global anisotropic Calder\'on problem in Theorem~\ref{Main1} which says that there exists  constants $C$ and $0<\theta<1$ such that $\| \alpha - \tilde{\alpha} \|_{C^{0,r}(M)} \leq C \left( \ln \frac{1}{\epsilon} \right)^{-\theta}$ for some $0<r<1$. Similar results are obtained in Theorem~\ref{Main2} for the partial anisotropic Calder\'on problem, corresponding to the case where the data are measured on only one connected component of the boundary. We emphasize the fact that even though we require the compatibility of the conformal factor $\alpha$ with the block-separability of the the Laplace-Beltrami operator $\Delta_{g_0}$ in our hypotheses, our transversal manifolds $(K,g_K)$ are closed manifolds with no restriction on the injectivity of the geodesic ray transform. In particular, the transversal manifold could be the round sphere which is the classical counterexample to injectivity of the geodesic ray transform. This implies that the powerful Complex Geometrical Optics techniques of \cite{DKSU2009, DKLS2016} cannot be applied in our setting. 

\end{abstract}

\tableofcontents


\vspace{1cm}

\noindent \textit{Keywords}. Anisotropic Calder\'{o}n inverse problem, Painlev\'e manifolds, moment problems, Weyl-Titchmarsh functions, Carleman estimates. \\


\noindent \textit{2010 Mathematics Subject Classification}. Primaries 81U40, 35P25; Secondary 58J50.


\tableofcontents


\Section{Introduction}
The anisotropic Calder\'on problem is an important classical inverse problem which consists in recovering, up to some natural gauge equivalences, the metric of a compact Riemannian manifold with boundary $(M,g)$ from the knowledge of the Dirichlet-to-Neumann (DN) map for the Laplace-Beltrami operator $-\Delta_g$ at some fixed frequency lying outside the Dirichlet spectrum (see \cite{U1, U2} for a general presentation of this classical problem). 

The question of \emph{stability}, namely whether one can quantify in a stable way the error in the solution of the inverse problem assuming that the DN map is only known up to an error term measured in an appropriate norm, is one of significant complexity due to the fundamentally ill-posed nature of the Calder\'on inverse problem (see for instance \cite{Ale1988, Ale1997, Ale2007, Man2001}). In this paper, we consider question of stability for a special class of Riemannian manifolds which are of geometric interest and for which the powerful methods based on the reconstruction of the metric via complex geometric optics solutions and limiting Carleman weights are not directly applicable. We now proceed to define this class of manifolds - that we will call \emph{Painlev\'e-Liouville Riemannian manifolds} - and state our main results.

Let $M$ denote an $n$-dimensional smooth compact orientable product manifold with boundary, of the form
\begin{equation} \label{M}
M=[0,1]\times K\,,
\end{equation}
where $K$ is a smooth closed connected orientable manifold of dimension $n-1$. The two connected components of the boundary of $M$ will be denoted by $\Gamma_0$ and $\Gamma_1$,
\[
\partial M=\Gamma_0 \sqcup \Gamma_1\,,\quad \Gamma_0=\{0\}\times K\,, \quad \Gamma_1=\{1\}\times K\,,
\]
and will be identified with $K$. The points of $M$ will be written as $(x,\omega)$, with $x$ and $\omega$ being referred to respectively as the radial and angular variables on $M$. We endow $M$ with a product Riemannian metric
\begin{equation} \label{g0}
g_{0}=dx^2+g_{K}\,,
\end{equation}
with $x \in [0,1]$ and $g_K$ being a Riemannian metric on $K$ and we introduce a conformal rescaling $g$ of $g_0$ given by
\begin{equation}\label{confPain}
g=\alpha^4(x,\omega)g_{0}\,.
\end{equation}

It is a classical result that the boundary value problem 
\begin{equation}\label{bvp}
  \left\{ \begin{split}
      -\Delta_{g}u=0  & \quad \text{on $M$}\\
      u=f & \quad \text {on $\partial M$}
       \end{split}  \right.     
\end{equation}
has a unique solution $u\in H^1(M)$ given boundary data $f \in H^{1/2}(M)$. When $f$ is smooth, the DN map associated to the boundary value problem \eqref{bvp} is then defined by
\begin{equation}\label{DNmapstrong}
	\Dn f = \partial_{\nu}u  \ |_{\partial M} \  , 
\end{equation}
where $\nu = (\nu^i)$ denotes the unit outer normal to the boundary $\partial M$. In the general case in which $f \in H^{1/2}(M)$, the DN map is defined in a weak sense as a map from $H^{1/2}(\partial M)$ to $H^{-1/2}(\partial M)$ by
\begin{equation}\label{weakDN}
	\left \langle \Dn f | h \right \rangle = \int_M g(\nabla u, \nabla v) \ dvol_g \, , \quad  \forall f,  h \in H^{1/2}(\partial M),
\end{equation}
where $u$ is the unique solution of \eqref{bvp}, $v$ is any element of $H^1(M)$ s.t. $v_{|\partial \Omega} = h$, and
$\left\langle \cdot  | \cdot  \right \rangle$ is the standard $L^2$ duality pairing between $ H^{1/2}(\partial M)$ and $ H^{-1/2}(\partial M)$. One shows that the DN map $\Dn$ is an elliptic pseudo-differential operator $\Dn$ of order $1$ that is selfadjoint on $L^2(\partial M)$. The DN map associated to the metric $g_0$ is defined similarly and will be denoted by $\Dno$. 

An observation that will play a key role in our analysis is that the boundary value problem \eqref{bvp} is equivalent upon setting 
\begin{equation}\label{rescale}
	v=\alpha^{n-2}u\,,\quad \psi=\alpha^{n-2}|_{\partial M}f\,,
\end{equation}
to the boundary value problem for a corresponding Schr\"odinger operator on $(M,g_0)$, given by
\begin{equation}\label{bvps}
	\left\{ \begin{split}
		(-\Delta_{g_0}+q)v =0, & \quad \text{on $M$},\\
		v =\psi, & \quad \text{on $\partial M$},
	\end{split} \right.       
\end{equation}
where 
$$
    q=\frac{\Delta_{g_0}\alpha^{n-2}}{\alpha^{n-2}}\,.
$$
This follows from the well-known fact that the positive Laplace-Beltrami operators $-\Delta_{g_0}$ and $-\Delta_{g}$ are related by the identity :
\[
-\Delta_{g}=\alpha^{-(n+2)}(-\Delta_{g_0}+q)\alpha^{n-2}\,,\quad q=\frac{\Delta_{g_0}\alpha^{n-2}}{\alpha^{n-2}}\,.
\]
Associated to \eqref{bvps}, we can thus define similarly the DN map $\Lambda_{g_0,q}$ as the operator from $H^{\frac{1}{2}}(\partial M)$ onto $H^{-\frac{1}{2}}(\partial M)$ by 
$$
  \Lambda_{g_0,q} \psi = \partial_\nu v_{|\partial M}.
$$
We observe that the anisotropic Calder\'on problems for the DN map $\Lambda_g$ and $\Lambda_{g_0,q}$ are essentially equivalent, a fact that will be used crucially in this paper.

We shall be likewise interested in the partial DN maps corresponding to the cases in which the Dirichlet data $f$ are supported in only one of the connected components $\Gamma_0$ or $\Gamma_1$ of $\partial M$ and the normal derivative of the solution $u$ of \eqref{bvp} is also only measured on one of the boundary components. These will be denoted by $\Dngammaij: H^{1/2}(K) \to H^{-1/2}(K)$, defined for $f$ smooth with $\mathrm{supp} f \subseteq \Gamma_i$ by
\[
\Dngammaij f = \partial_{\nu}u  \ |_{ \Gamma_{j}} \,,
\]
with a definition in the weak sense similar to \eqref{weakDN} when $f \in H^{1/2}(K)$. Finally, we will also consider the DN maps $\Dngammaqij$ associated to the boundary value problem \eqref{bvps} for the corresponding Schr\"odinger operator. 

In this paper, we shall consider the stability problem for two metrics $g, \tilde{g}$ of the form (\ref{confPain}) under the hypothesis of equality of the angular metrics, that is  
$$
  g_{K}={\tilde{g}}_{K}\,. 
$$
Our task will be thus to obtain stability estimates on the conformal factor $\alpha$, assuming that for some $0<\epsilon<1$
\begin{equation}\label{DNMapHyp-global}
\lVert \Lambda_{g}-\Lambda_{\tilde{g}}\rVert_{\mathcal{B}(H^{1/2}(\partial M), H^{-1/2}(\partial M))}\ = \epsilon\,, 
\end{equation}
or one of the weaker conditions
\begin{equation} \label{DNMapHyp-partial} 
\lVert \Dngammai - \Dngammaitilde \rVert _{{\cal{B}}(H^{1/2}(\Gamma_i),H^{-1/2}(\Gamma_i))} = \epsilon\,,\quad i \in \{0,1\}\,.
\end{equation}
Precisely, under the assumption (\ref{DNMapHyp-global}) (resp. (\ref{DNMapHyp-partial})) corresponding to a global (resp. partial) Calder\'on problem, we aim to prove an estimate of the form
\begin{equation} \label{Aim}
  \| \alpha - \tilde{\alpha} \|_{L^\infty(M)} \leq C f(\epsilon), 
\end{equation}
for a function $f$ satisfying $\ds\lim_{\epsilon \to 0} f(\epsilon) = 0$.

In order to obtain (and quantify) such stability estimates, we impose the following conditions. 

\begin{hypo} \label{Hyp}

i) $g_K = \tilde{g}_K$. \\
ii) There exists a contant $c > 0$ such that  $\alpha, \tilde{\alpha} \geq c$ on $M$. \\
iii) There exist constants $0<r<1$ and $B > 0$ for which we have $\alpha, \tilde{\alpha} \in C^{2,r}(M)$ and $\|\alpha\|_{C^{2,r}}, \|\tilde{\alpha}\|_{C^{2,r}} \leq B$. \\
iv) There exist $\phi_1 = \Phi_1(x), \phi_2 = \Phi_2(\omega)$ (resp. $\tilde{\phi_1} = \tilde{\Phi_1}(x), \tilde{\phi_2} = \tilde{\Phi_2}(\omega)$) such that 
\begin{equation} \label{Separability}
  q = \frac{\Delta_{g_0} \alpha^{n-2}}{\alpha^{n-2}} = \phi_1(x) + \phi_2(\omega), \quad \tilde{q} = \frac{\Delta_{g_0} \tilde{\alpha}^{n-2}}{\tilde{\alpha}^{n-2}} = \tilde{\phi}_1(x) + \tilde{\phi}_2(\omega). 
\end{equation}
\end{hypo}

Let us make some remarks on the contents of Hypothesis \ref{Hyp}. The assumption $i)$ is made to simplify the analysis and obtain better stability estimates of logarithmic type (see below). In a companion paper \cite{DKN2026}, we study the general case of warped products manifolds whose warping function $\alpha = \alpha(x)$ and angular metric $g_K$ are both allowed to vary. The assumptions $ii)$ and $iii)$ are classical and reflect the fact that the Calder\'on problem is ill-posed. This means that the inverse of the map $g \mapsto \Lambda_g$ is not continuous unless the metric $g$ (here the conformal factor $\alpha$) lies in a compact set \cite{Ale1988}. Moreover the modulus of continuity is generally, at best, logarithmic \cite{Man2001, KRS2021}. The assumption $iv)$ is the most important and more restrictive one. Roughly speaking, (\ref{Separability}) means that the conformal factor $\alpha^4$ is compatible with the Painlev\'e type block separability of the metric \eqref{confPain}. As a consequence, the DN maps $\Lambda_g, \Lambda_{\tilde{g}}$ will be "diagonalizable" on well chosen Hilbert bases $(Y_k)_{k \geq 0}$ and $(\tilde{Y_k})_{k \geq 0}$, a fact that will be crucially used in our analysis. We refer to \cite{DKN5} for more informations on Painlev\'e block separability in full generality and to section \ref{DN} for the simple instance of block separability used in this paper. Finally, we will show in Appendix \ref{A} that we can  easily construct a large class of conformal factors $\alpha$ satisfying all the above assumptions.

\begin{rems} \label{ConstantA}
1. Under assumptions ii), iii) and iv), it is clear that there exists a constant $A = A(g_0, c, B)$ such that $\phi_1, \phi_2$ satisfy the assumptions :
\begin{equation} \label{A}
	\phi_1 \in C^{0,r}([0,1]), \ \phi_2 \in C^{0,r}(K), \quad \|\phi_1\|_\infty, \ \|\phi_2\|_\infty \leq A. 
\end{equation}	  
2. The functions $\phi_1, \phi_2$ aren't unique in the block separability property \eqref{Separability} since for any constant $c$, we could replace $\phi_1$ by $\phi_1 + c$ and $\phi_2$ by $\phi_2 - c$. Using this gauge, we will assume from the very beginning that 
\begin{equation} \label{Phi_1(1)}
	\phi_1(j) = \tilde{\phi}_1(j),
\end{equation}	
for $j=0$ or $j=1$ This assumption will simplify slightly the forthcoming analysis. 
\end{rems}

Motivated by the above hypotheses, we define the class of Painlev\'e-Liouville Riemannian manifolds as follows:
\begin{defi}
We say that $(M,g)$ is a Painlev\'e-Liouville Riemannian manifold if 
$$
M=[0,1]\times K\,,\quad g=\alpha^4(x,\omega)g_{0}\,,\quad g_{0}=dx^2+g_{K}\,,
$$
where we have $\alpha>c$ for some positive constant $c$ and where  
\begin{equation}\label{pdealpha}
\frac{\Delta_{g_0} \alpha^{n-2}}{\alpha^{n-2}} = \phi_1(x) + \phi_2(\omega)\,.
\end{equation}
\end{defi}
In Appendix \ref{A}, we address the important question of the existence of a solution $\alpha>0$ of the elliptic equation \eqref{pdealpha}, a question that is crucially relevant to the actual existence of Painlev\'e-Liouville Riemannian manifolds. 
 
Our main stability result for the corresponding \emph{global} Calder\'on problem can now be stated as follows. 

\begin{thm} \label{Main1}
Let $(M,g)$ and $(M,\tilde{g})$ be two Painlev\'e-Liouville Riemannian manifolds as in \eqref{confPain} satisfying the hypotheses \ref{Hyp}. Assume that \eqref{DNMapHyp-global} holds. Then \\
i) [\textbf{H\"older stability at the boundary $\partial M$}] There exists a constant $C = C(g_0,c,B) > 0$ such that for all $j=0,1$, $k=0,1,2$, there exists $0<\theta_k\leq1$ :
\begin{equation} \label{Angular}
	\| \alpha(j,.) - \tilde{\alpha}(j,.) \|_{C_E^k(\Gamma_j)} \leq C \epsilon^{\theta_k},
\end{equation}	
where the notation $C^k_E(\Gamma_j)$ means that the tangential and normal derivatives at $\Gamma_j$ are included in the definition of the $C^k$ norm. \\

\noindent ii) [\textbf{Logarithmic stability within $M$}] There exists a constant $C = C(g_0,c,B)$ and $0<\theta<1$ such that :
\begin{equation} \label{Radial}
	\| \alpha - \tilde{\alpha} \|_{C^{0,r}(M)} \leq C \left( \ln \frac{1}{\epsilon} \right)^{-\theta}. 
\end{equation}	 
\end{thm}

We note that the first H\"older stability result at the boundary is an easy consequence of the classical boundary determination results obtained (in the convenient form for us) by Kang and Yun in \cite{KY2002}. The second logarithmic stability result within the manifold $M$ is the new part in this theorem. 

We also obtain stability results for the \emph{partial} Calder\'on problem where the measurements are made either on $\Gamma_0$ or $\Gamma_1$. Roughly speaking, assuming that the measures are made on $\Gamma_1$, we obtain logarithmic stability estimates on any regions that do not contain a neighbourhood of the inaccessible part $\Gamma_0$. Precisely, we prove : 

\begin{thm} \label{Main2}
	Let $(M,g)$ and $(M,\tilde{g})$ be two metrics as in \eqref{confPain} satisfying the hypotheses \ref{Hyp}. Assume that \eqref{DNMapHyp-partial} holds at $\Gamma_j, \ j=0$ or $1$ and additionally that there exist constants $D>0$ and $0<\nu<1$ with 
	\begin{equation} \label{MoyennePhi}
	  \left| \int_0^1(\phi_1(x) - \tilde{\phi}_1(x)) dx \right| \leq D \epsilon^\nu.
	\end{equation}
	Then \\
	1. [\textbf{Hölder stability at the boundary $\Gamma_j$}] There exists a constant $C = C(g_0,c,B,D) > 0$ such that for all $k=0,1,2$, there exists $0<\theta_k\leq1$ :
	\begin{equation} \label{Angular-partial}
		\| \alpha(j,.) - \tilde{\alpha}(j,.) \|_{C_E^k(\Gamma_j)} \leq C \epsilon^{\theta_k},
	\end{equation}	
	where the notation $C^k_E(\Gamma_j)$ means that the tangential and normal derivatives at $\Gamma_j$ are included in the definition of the $C^k$ norm. \\
	
	\noindent 2. [\textbf{Logarithmic stability within $M$}] Assume that the measures are made on $\Gamma_1$. For any $0<\tau<1$, define $M_\tau = [\tau,1] \times K$. Then there exists constants $C_\tau = C(g_0,c,B,D,\tau)$ and $0<\theta_\tau<1$ such that :
	\begin{equation} \label{Radial_partial}
		\| \alpha - \tilde{\alpha} \|_{C^{0,r}(M_\tau)} \leq C_\tau \left( \ln \frac{1}{\epsilon} \right)^{-\theta_\tau}, 
	\end{equation}	 
    where $C_\tau \lesssim \tau^{-2}$ and
    \[
    \theta_\tau
    =
    \frac{2(1-r)}{n}\,
    \frac{\theta\,\tau}{2-\tau}.
    \]
    Symmetric results hold when the measures are made on $\Gamma_0$. 
\end{thm}

\begin{rems}
1. The assumption \eqref{MoyennePhi} can be re-expressed as a condition on the scalar curvatures $Scal_g, Scal_{\tilde{g}}$. More precisely, it is well known that 
$$
  q = \frac{n-2}{4(n-1)} \left( \alpha^4 Scal_g - Scal_{g_0} \right). 
$$
From this, \eqref{MoyennePhi} can be shown to be equivalent to the condition :
\begin{equation} \label{MoyenneCurvature}
  \left| \int_0^1 [\alpha^4 Scal_g - \tilde{\alpha}^4 Scal_{\tilde{g}}] dx \right| \leq D \epsilon^\nu.	
\end{equation}	
Moreover, we emphasize the fact that the additional assumption \eqref{MoyennePhi} in the previous partial stability theorem is a consequence of the assumption \eqref{DNMapHyp-global} used in the global stability theorem. \\

\noindent 2. In the estimate \eqref{Radial_partial}, we can allow $\tau$ to depend on $\epsilon$ in such a way that $\tau(\epsilon) \to 0$. We extend this way (and in the limit $\e \to 0$) our stability estimates to bigger and bigger regions $M_{\tau(\e)}$ that still all avoid the inaccessible part $\Gamma_0$ at the price of weakening the stability rate. For instance, we get with the choice
\[
  \tau(\varepsilon) = \frac{1}{\sqrt{\ln \, \ln(1/\varepsilon)}}.
  \]
Then
\[
	\| \alpha - \tilde{\alpha}\|_{C^{0,r}(M_{\tau(\varepsilon)})}
	\;\lesssim\;
	\ln \, \ln(1/\varepsilon)\,
	\exp\,\Bigl(
	-\frac{(1-r)\theta}{n}\,
	\sqrt{\ln \, \ln(1/\varepsilon)}
	\Bigr), \quad \varepsilon \to 0.
\]
Thus we obtain a stability rate which is stronger than the global log--log estimates available for Calder\'on-type inverse problems with partial data (see, e.g., Caro--Dos Santos Ferreira--Ruiz~\cite{CDR2016}).
\end{rems}

Let us now compare our results with the existing litterature. First, assume that $(K,g_K) = (\mathbb{S}^{n-1}, d\omega^2)$ is the $(n-1)$-dimensional sphere equipped with the round metric. Setting $\beta(x,\omega) = \alpha(x,\omega) e^{\frac{x}{2}}$ and introducing the new coordinate $r = e^{-x}$, the Riemannian manifold $(M,g)$ can be expressed as : 
$$
  M = [e^{-1}, 1]\times \mathbb{S}^{n-1}\,, \quad g = \beta^4(r,\omega) [dr^2 + r^2 d\omega^2],
$$
and can thus be viewed as a conformal deformation of a $n$-dimensional annulus equipped with the euclidean metric. For such models that correspond to the classical Calder\'on problem for isotropic conductivities, Alessandrini was the first to obtain logarithmic stability estimates in \cite{Ale1988}. This result was later improved in several directions for instance by Novikov \cite{Nov2011} for the stability of the Gelfand-Calder\'on problem (\textit{i.e.} the Calder\'on problem for a Schrödinger equation), by Alessandrini and Gaburro who obtained stability estimates for \emph{anisotropic} conductivities having a certain form (but still depending on a scalar function) in \cite{AlGa2001} and studied the corresponding partial Calderon problem in \cite{AlGa2009} and finally, by Heck and Wang, and Caro, Dos Santos Ferreira and Ruiz, who studied the stability of the \emph{partial} Calder\'on problem for isotropic conductivities in \cite{HeWa2016} and \cite{CDR2016} respectively. In the three latter references, only log log-type stability estimates were obtained for the partial Calder\'on problem, except for situations in which the inaccessible part of the boundary is either a plane or part of a sphere. In that case, log-type stability estimates exist. Our logarithmic stability results extend thus these classical results to the class of Painlev\'e-Liouville Riemannian manifolds  for both the global and partial Calder\'on problems where the conductivity $\alpha$ satisfies \eqref{Separability}. Note however that in the latter case, we only obtain logarithmic stability estimates on Painlevé-Liouville Riemannian manifolds from which we remove an arbitrary small neighbourhood of the inacessible part of the boundary. 

Second, closer to our models are the so-called admissible manifolds introduced in \cite{DKSU2009, DKLS2016} which are Riemannian manifolds $(M,g)$ such that 
\begin{equation} \label{CTA}
  M \subset \subset \R \times L, \quad g = c(x,\omega) [dx^2 + g_L], 
\end{equation}
where $c>0$ is a positive function with no restriction and the transversal manifold $(L,g_L)$ is a $(n-1)$-dimensional compact Riemannian manifold with boundary such that the geodesic ray transform is injective \footnote{Examples of such Riemannian manifolds are : (a) simple manifolds of any dimension \cite{Sha1994}, (b) negatively curved manifolds with strictly convex boundary \cite{Gui2017}, (c) manifolds of dimension $\geq 3$ that have strictly convex boundary and are globally foliated by strictly convex hypersurfaces \cite{UhVa2016, PSUZ2019}.}. As shown in \cite{DKSU2009}, models satisfying \eqref{CTA} - which are called Conformally Transversally Anisotropic anifolds (\textbf{CTA}) in \cite{DKLS2016} - are the most general Riemannian manifolds for which Complex Geometrical Optics (CGO) solutions can be contructed - through the existence of limiting Carleman weights - and used to study uniqueness and stability in the anisotropic Calder\'on problem. Relevant for this paper are the log log-type stability estimates for the conformal factor $c$ that have been obtained by Caro and Salo in \cite{CaSa2014}. 

Our model differ from admissible manifolds with respect to three important points : 
\begin{itemize}
\item Our transversal manifolds $(K,g_K)$ are closed manifolds with no restriction on the injectivity of the geodesic ray transform. In particular, the transversal manifold could be the round sphere which is the classical counterexample to injectivity of the geodesic ray transform. In consequence, the classical CGO techniques that amount ultimately to inverting a geodesic ray transform cannot be used in our setting. However, our Painlevé-Liouville manifolds are CTA manifolds satisfying \eqref{CTA} and thus limiting Carleman weights exist on them. We will use these classical Carleman estimates in the version given in  \cite{KS2013} to prove our stability estimates in the \emph{partial} Calder\'on problem. 

\item However, our conformal factors $\alpha$ satisfy the separability condition \eqref{Separability} and so, are much more restricted than the conformal factors $c$ allowed in admissible geometries. They depend essentially on the two (almost) arbitrary functions $\phi_1(x)$ and $\phi_2(\omega)$. This will be shown rigorously in Appendix \ref{A}. 

\item The boundaries of our manifolds $(M,g)$ have two connected components that are compatible with separation of variables in opposition to the general connected boundaries used in admissible geometries. The resulting possibility to diagonalize the DN map on a well chosen Hilbert basis of eigenfunctions  is at the heart of our analysis. We refer to Section \ref{DN} for this point. 
\end{itemize}

Despite these differences between admissible and Painlev\'e-Liouville manifolds which make difficult the comparison, we observe that our results improve the log log-type stability estimate obtained in \cite{CaSa2014} to log-type stability estimates in cases where the CGO techniques cannot be applied.  

Let us finish this introduction with the outline of this paper. First, in Section \ref{DN}, we use the particular geometry of Painlev\'e-Liouville manifolds (that are conformal to a product metric) in order to express the DN map $\Lambda_g$ in terms of the DN map $\Lambda_{g_0, q}$. Thanks to the separability condition \eqref{Separability}, the latter can be diagonalized onto a Hilbert basis $(Y_k)_{k\geq0}$ corresponding to a choice of normalized eigenfunctions of the Schr\"odinger operator $-\Delta_{g_K} + \phi_2$ where $-\Delta_{g_K}$ is the Laplace-Beltrami operator for the transversal manifolds $(K,g_K)$. We are thus led to study $1$-dimensional DN maps associated to the Schr\"odinger equations 
\begin{equation} \label{Schro}
  -v'' + \phi_1(x) v = -\mu_k^2 v, \quad x \in [0,1],
\end{equation}  
parametrized by the eigenvalues $(\mu^2_k)_{k \geq 0}$ of $-\Delta_{g_K} + \phi_2$. More precisely, using the fact that the boundary can be identified with two copies of $K$, we prove that the reduced DN map onto each subspace spanned by $(Y_k)$ is a multiplication operator by a $2\times2$ matrix given by
$$
  (\Lambda_{g_0,q})_{|\langle Y_k \rangle} = \begin{pmatrix}
  -M(-\mu_{k}^{2}) & -\frac{1}{\Delta(-\mu_{k}^{2})}\\
  -\frac{1}{\Delta(-\mu_{k}^{2})} & -N(-\mu_{k}^{2})
  \end{pmatrix},
$$
where $M,N$ and $\Delta$ are the Weyl-Titchmarsh and characteristic functions associated to \eqref{Schro} respectively. This will allow us to use powerful tools from $1$-dimensional inverse spectral theory. Second, in Section \ref{BD}, we use the classical boundary determination results from \cite{KY2002} to prove the Hölder stability at the boundary stated in Theorems \ref{Main1} and \ref{Main2}. From this result and the separability condition \eqref{Separability}, we can then easily deduce that there exists $0<\theta<1$ such that :
$$
  \| \phi_2 - \tilde{\phi}_2 \|_\infty \leq C \epsilon^\theta, \quad \forall k \geq 0, \quad  |\mu_k^2 - \tilde{\mu}_k^2| \leq C \epsilon^\theta. 
$$
Third, in Section \ref{GlobalStability}, we first prove some preliminary results on the Weyl-Titchmarsh functions $M,N$ and the characteristic function $\Delta$. In particular, using our main hypotheses \eqref{DNMapHyp-global}, we are able to prove that 
\begin{equation} \label{MD}
  \forall k \geq 0, \quad |M(-\mu_k^2) - \tilde{M}(-\tilde{\mu}_k^2) | \leq C \epsilon^\theta, \quad |\Delta(-\mu_k^2) - \tilde{\Delta}(-\tilde{\mu}_k^2) | \leq C \epsilon^\theta
\end{equation}
Then, we start from the identity
\begin{multline*}
  \int_0^1 [\phi_1 - \tilde{\phi}_1 + \mu_k^2 - \tilde{\mu}_k^2] c_0(x,\mu_k) \tilde{s}_0(x,\tilde{\mu}_k) dx \\ =  M(-\mu_k^2) \Delta(-\mu_k^2) \tilde{\Delta}(-\tilde{\mu}_k^2) \left( \tilde{N}(-\tilde{\mu}_k^2) - N(\mu_k^2) \right) + \frac{\tilde{\Delta}(-\tilde{\mu}_k^2) - \Delta(-\mu_k^2)}{\Delta(-\mu_k^2)}, 
\end{multline*}
that comes essentially from the Schr\"odinger equation. Here, $c_0, s_0$ are solutions of \eqref{Schro} with $\cos$ and $\sin$ type boundary conditions at $0$. Next, using the classical transformation operators associated to \eqref{Schro} (see for instance \cite{Mar2011}) and the previous estimates \eqref{MD}, we can prove (following Gendron \cite{Gen2022} with additional technical work) that there exists a Lipschitz invertible operator $R : L^2(0,2) \longrightarrow L^2(0,2)$ such that 
\begin{equation} \label{Moment}
 \forall k \geq 0, \quad  \left|  \int_0^{2} e^{-2\mu_k x} R[\phi_1 - \tilde{\phi}_1](x) dx \right| \leq C \epsilon^\theta. 
\end{equation}
This is equivalent to a moment problem and we can use the same method as \cite{DKN7, DKN8, Gen2022} to prove that
$$
  \| \phi_1 - \tilde{\phi}_1 \|_2 \leq C \left( \ln(\frac{1}{\epsilon})\right)^{-\theta}. 
$$
Note that the well-known ill-posedness of classical moment problems is the origin for the loss of stability (from Hölder to logarithmic) within the manifold. As a conclusion, putting all the previous results together, we have proved at this stage that 
$$
  \| q - \tilde{q} \|_2 \leq C \left( \ln(\frac{1}{\epsilon})\right)^{-\theta}. 
$$
Finally, using the particular form of the potentials $q, \tilde{q}$, we can prove that $\ln \alpha - \ln \tilde{\alpha}$ satisfies an elliptic PDE whose inhomogenous term is given by $q - \tilde{q}$. Using a standard existence result for elliptic PDE and a trick given in \cite{CGR2013}, we can prove our main inner stability estimate
$$
  \| \alpha - \tilde{\alpha} \|_{L^\infty(M)} \leq C \left( \ln(\frac{1}{\epsilon})\right)^{-\theta}. 
$$  

In Section \ref{PartialStability}, we study the stability problem for the partial Calder\'on problem where the measurements are made only on $\Gamma_0$ or $\Gamma_1$. The strategy is the same as before, but we need to impose the extra condition \eqref{MoyennePhi} in order to obtain logarithmic stability estimates for $\|q - \tilde{q}\|_{L^2(M)}$. At last, the logarithmic  stability estimate for $\|\alpha - \tilde{\alpha}\|_{L^\infty(M_\tau)} \|$ is a consequence of the classical Carleman estimates that hold on general CTA manifolds. Note that the Dirichlet and Neumann data are measured on the same connected components $\Gamma_0$ or $\Gamma_1$ in this result. If the Dirichlet and Neumann data would be measured on different (and thus disjoint) connected components, then there is no uniqueness in the partial Calder\'on problem as shown in \cite{DKN2, DKN3, DKN4}. 

Finally, we finish this paper by Appendix \ref{A} containing a description of the class of conformal factors $\alpha$ satisfying Hypothesis \ref{Hyp} and Appendix \ref{B} in which we prove the Carleman estimates we use in the stability estimate for the partial Calder\'on problem.


\Section{The Dirichlet-to-Neumann map} \label{DN}
In this section, we exploit the separation of variables for the boundary value problems \eqref{bvp} and \eqref{bvps} to compute the DN map $\Lambda_{g}$. Our hypothesis that the conformal factor $\alpha^4(x,\omega)$ is compatible with the Painlev\'e-type block separability of the metric \eqref{confPain} is equivalent to the potential $q$ in \eqref{bvps} satisfying the block-separability condition
\[
q(x,\omega)=\phi_1(x)+\phi_2(\omega)\,,
\]
so that the Schr\"odinger equation in \eqref{bvps} reads
\begin{equation}\label{Sepop}
(-\partial_{x}^{2}+\phi_1(x)-\Delta_K+\phi_2(\omega))v=0\,,
\end{equation}
where $-\Delta_K$ denotes the positive Laplace-Beltrami operator on $(K,g_K)$. We denote the normalized eigenfunctions of the angular operator $-\Delta_K+\phi_2(\omega)$ by 
$Y_{k}(\omega)\,,\,k\geq 0\,,$ and the corresponding eigenvalues by $\mu_{k}^{2}$, so that 
\[
(-\Delta_K+\phi_2)Y_{k}=\mu_{k}^2 Y_{k}\,.
\]
The eigenfunctions $\big(Y_k(\omega)\big)_{k\geq 0}$ form a Hilbert basis of $L^2(K,dvol_{g_{K}})$ and we therefore seek $v$ in the form
\begin{equation}\label{Fourierv}
v(x,\omega)=\sum_{k\geq 0}v_{k}(x)Y_{k}{(\omega)}\,.
\end{equation}
Substituting \eqref{Fourierv} into \eqref{Sepop} and expanding the boundary data $\psi$ in \eqref{bvps} on the connected components $\{x=0\}$ and $\{x=1\}$ of $\partial M$ in the basis $\big(Y_k(\omega)\big)_{k\geq 0}$ as 
\[
\psi |_{x=0}(\omega)=:\psi_{0}(\omega)=\sum_{k\geq 0}\psi_{0,k}\,Y_{k}(\omega)\,,\quad \psi |_{x=1}(\omega)=:\psi_{1}(\omega)=\sum_{k\geq 0}\psi_{1,k}\,Y_{k}(\omega)\,,
\]
we obtain the sequence of radial boundary value problems given for $k\geq 0$ and $x\in [0,1]$ by
\begin{equation}\label{seqode}
\begin{cases}
-v_{k}{''}(x)+\phi_{1}(x)v_{k}(x)=-\mu_{k}^{2}v_{k}(x)\,,\\
v_{k}(0)=\psi_{0,k}\,,\quad v_{k}(1)=\psi_{1,k}\,.
\end{cases}
\end{equation}
We now consider fundamental systems of solutions $\{c_{0}(x,z),s_{0}(x,z)\}$ and $\{c_{1}(x,z),s_{1}(x,z)\}$ of the ODE
\begin{equation}\label{radialz}
-v{''}(x)+\phi_{1}(x)v(x)=z\,v(x)\,,\quad z\in \mathbb{C}\,,
\end{equation}
defined respectively by the Cauchy conditions 
\begin{equation}\label{Cauchy1}
c_{0}(0,z)=1\,,\,c'_{0}(0,z)=0\,,\quad s_{0}(0,z)=0\,,\,s'_{0}(0,z)=1\,,
\end{equation}
and 
\begin{equation}\label{Cauchy2}
c_{1}(1,z)=1\,,\,c'_{1}(1,z)=0\,,\quad s_{1}(1,z)=0\,,\,s'_{1}(1,z)=1\,.
\end{equation}
The characteristic function $\Delta(z)$ and the Weyl-Titchmarsh functions $M(z),N(z)$ corresponding to the o.d.e. \eqref{radialz} with Dirichlet boundary conditions $v(0)=0,v(1)=0$ are then defined in terms of the Wronskians of the elements of the fundamental sets by
\[
\Delta(z)=W(s_0,s_1)(z)\,,
\]
and 
\[
\,M(z)=-\frac{W(c_{0},s_{1})(z)}{\Delta(z)}=-\frac{c_{0}(1,z)}{s_{0}(1,z)}\,,\quad N(z)=-\frac{W(c_{1},s_{0})(z)}{\Delta(z)}=\frac{c_{1}(0,z)}{s_{1}(0,z)}\,.
\]
We expand in the basis $\big(Y_k(\omega)\big)_{k\geq 0}$ the DN map $\Lambda_{g_{0},q}$ for the boundary value problem \eqref{bvps} as
\begin{equation}\label{DNqFourier}
\Lambda_{g_{0},q}\begin{pmatrix}
            \psi_{0}\\
            \psi_{1}
         \end{pmatrix}
=\sum_{k\geq 0}\bigg(\Lambda^{k}_{g_{0},q}\begin{pmatrix}
            \psi_{0,k}\\
            \psi_{1,k}
         \end{pmatrix}
\bigg)Y_{k}\,,
\end{equation}
where 

  \begin{align}\label{DNqFourierCoeff}
    \Lambda^{k}_{g_{0},q} \begin{pmatrix}
            \psi_{0,k}\\
            \psi_{1,k}
         \end{pmatrix}
 &= \begin{pmatrix}
           -v_{k}'(0) \\
            v_{k}'(1)
         \end{pmatrix} \,.
  \end{align}
The coefficients \eqref{DNqFourierCoeff} may now be expressed in terms of the characteristic function $\Delta$ and Weyl-Titchmarsh functions $M,N$ evaluated at $-\mu_{k}^{2}$ (where the $\mu_{k}^{2}$ are the eigenvalues of the angular operator $-\Delta_K+\phi_2(\omega)$) by a straightforward calculation using \eqref{DNqFourier} and the Cauchy conditions \eqref{Cauchy1}, \eqref{Cauchy2}, giving 
  \begin{align}\label{DNqFouriermat}
    \Lambda^{k}_{g_{0},q} \begin{pmatrix}
            \psi_{0,k}\\
            \psi_{1,k}
         \end{pmatrix}
 &= \begin{pmatrix}
           -M(-\mu_{k}^{2}) & -\frac{1}{\Delta(-\mu_{k}^{2})}\\
            -\frac{1}{\Delta(-\mu_{k}^{2})} & -N(-\mu_{k}^{2})
         \end{pmatrix} \begin{pmatrix}
            \psi_{0,k}\\
            \psi_{1,k}
         \end{pmatrix}
 \,.
  \end{align}
Note that the diagonal components $ -M(-\mu_{k}^{2})$ and $-N(-\mu_{k}^{2})$ of the matrix \eqref{DNqFouriermat} are respectively the eigenvalues of the partial DN maps $\Dngammaqzero$ and $\Dngammaqone$. We will use this simple observation together with min-max characterization in Section \ref{GlobalStability}

Our final step is to compute the DN map $\Lambda_{g}$ for our original  boundary value problem \eqref{bvp} for the Laplacian $\Delta_{g}$ in terms of the DN map $ \Lambda_{g_{0},q} $ for the boundary value problem \eqref{bvps}, whose expression we have just obtained in \eqref{DNqFourier} and \eqref{DNqFourierCoeff}. Letting
\[
f|_{x=0}(\omega)=:f_0(\omega)\,,\quad f|_{x=1}(\omega)=:f_1(\omega)\,,
\]
we have 
\begin{align}
    \Lambda_{g} \begin{pmatrix}
            f_{0}(\omega)\\
            f_{1}(\omega)
         \end{pmatrix}
 &= \begin{pmatrix}
         -\alpha^{-2}(0,\omega)\,\partial_{x}u(0,\omega)\\
           \alpha^{-2}(1,\omega)\,\partial_{x}u(1,\omega)
         \end{pmatrix} 
 \,.
  \end{align}
From \eqref{rescale} we have $u=\alpha^{2-n}v$, so that using \eqref{DNqFourierCoeff} we obtain
\begin{align}
    \Lambda_{g} \begin{pmatrix}
            f_{0}(\omega)\\
            f_{1}(\omega)
         \end{pmatrix}
 &= \begin{pmatrix}
        -(2-n) \alpha^{-3}(0,\omega)\,\partial_{x}\alpha(0,\omega)f_{0}(\omega)-\alpha^{-n}(0,\omega)\,\partial_{x}v(0,\omega)\\
          (2-n) \alpha^{-3}(1,\omega)\,\partial_{x}\alpha(1,\omega)f_{1}(\omega)+\alpha^{-n}(1,\omega)\,\partial_{x}v(1,\omega)
         \end{pmatrix} 
 \,.
  \end{align}
 Letting
  \begin{equation}\label{defB1}
B_1=\begin{pmatrix}
            -(2-n) \alpha^{-3}(0,\omega)\,\partial_{x}\alpha(0,\omega) & 0\\
          0 &   (2-n) \alpha^{-3}(1,\omega)\,\partial_{x}\alpha(1,\omega)
         \end{pmatrix}\,,
\end{equation}
and 
\begin{equation}\label{defB2}
B_2=\begin{pmatrix}
            \alpha^{-n}(0,\omega) & 0\\
          0 &  \alpha^{-n}(1,\omega)         
           \end{pmatrix}\,,
\end{equation}
this gives
\begin{align}
    \Lambda_{g} \begin{pmatrix}
            f_{0}(\omega)\\
            f_{1}(\omega)
         \end{pmatrix}
           &= B_1 \begin{pmatrix}
            f_{0}(\omega)\\
            f_{1}(\omega)
         \end{pmatrix} + B_2\,\Lambda_{g_{0},q}\begin{pmatrix}
            \psi_{0}(\omega)\\
            \psi_{1}(\omega)
         \end{pmatrix}\,.
  \end{align}
 Therefore, we have proved using \eqref{rescale}:
 \begin{lemma}\label{DNExpr}
  \begin{align}
    \Lambda_{g} \begin{pmatrix}
            f_{0}(\omega)\\
            f_{1}(\omega)
         \end{pmatrix}
           &= B_1 \begin{pmatrix}
            f_{0}(\omega)\\
            f_{1}(\omega)
         \end{pmatrix} + B_2\,\Lambda_{g_{0},q}\begin{pmatrix}
            \alpha^{n-2}(0,\omega)f_{0}(\omega)\\
               \alpha^{n-2}(0,\omega) f_{1}(\omega)
         \end{pmatrix}\,,
  \end{align}
  where $\Lambda_{g_{0},q}$ is the $2 \times 2$ matrix multiplication operator given by the Fourier expansion \eqref{DNqFourier} and \eqref{DNqFouriermat}, and where $B_1$ and $B_2$ are the $2 \times 2$ matrix multiplication operators given by \eqref{defB1} and  \eqref{defB2} respectively.
  \end{lemma}

\Section{Boundary determination results and applications to angular stability estimates} \label{BD}

In this section, we prove the Hölder stability at the boundary stated in Theorems \ref{Main1} and \ref{Main2} as well as some consequences. 

Under the assumptions $i), ii)$ and $iii)$ of Hypothesis \ref{Hyp} and the main assumption \eqref{DNMapHyp-global} or resp. \eqref{DNMapHyp-partial}, we obtain as a direct application of \cite{KY2002} (Thm 1.3)
\begin{equation}\label{BDg}
\lVert \alpha^4(j,\omega)g_{K} - \tilde{\alpha}^4(j,\omega)  \tilde{g}_{K} \rVert_{C_E^{k}(\Gamma_j)} \leq C \epsilon^{\theta_k}\,, 
\end{equation}
where $j \in \{0,1\}$, $k\in\{0,1,2\}$, $\theta_k = 2^{-lk}$ for a given large enough integer $l \in \mathbb{N}$ and $C = C(K,g_K,c,B)$ is a constant. Recall here that the notation $C^k_E(\Gamma_j)$ means that the tangential and normal derivatives at $\Gamma_j$ are included in the definition of the $C^k$ norm. Since  $g_K$ is a fixed Riemannian metric, we obtain directly from \eqref{BDg} :
\begin{equation} \label{BDalpha}
  \lVert \alpha(j,\omega) - \tilde{\alpha}(j,\omega) \rVert_{C_E^{k}(K)} \leq C \epsilon^{\theta_k}, \quad k=0,1,2, 
\end{equation}
which is essentially \eqref{Angular} and \eqref{Angular-partial}. As a first corollary, we obtain :
\begin{coro} \label{DNq}
Under the assumption \eqref{DNMapHyp-global} (resp. \eqref{DNMapHyp-partial}), we have : 
$$
  \| \Lambda_{g_0,q} - \Lambda_{g_0,\tilde{q}} \|_{\mathcal{B}(H^{1/2}(\partial M), H^{-1/2}(\partial M))} \leq C \epsilon^{\theta_1}, 
$$
$$  
   (\mathrm{resp.} \quad \| \Lambda_{g_0,q,\Gamma_j, \Gamma_j} - \Lambda_{g_0,\tilde{q},\Gamma_j,\Gamma_j} \|_{\mathcal{B}(H^{1/2}(\Gamma_j), H^{-1/2}(\Gamma_j))} \leq C \epsilon^{\theta_1} ).  
$$  
\end{coro}
\begin{proof}
This is a direct consequence of Lemma \ref{DNExpr}, \eqref{DNMapHyp-global} or \eqref{DNMapHyp-partial} and \eqref{BDalpha}. Indeed, from Lemma \ref{DNExpr}, we first have
$$
\big( \Lambda_{g_{0},q}-\Lambda_{g_{0},\tilde{q}}\big)\psi=B_{2}^{-1}\Lambda_{g}\big( \frac{\psi}{\alpha^{n-2}} \big)-B_{2}^{-1}B_{1}\big( \frac{\psi}{\alpha^{n-2}} \big)-{\tilde{B_{2}}}^{-1}\Lambda_{\tilde g}\big( \frac{\psi}{\tilde{\alpha}^{n-2}} \big)+{\tilde{B_{2}}}^{-1}{\tilde{B_{1}}}\big( \frac{\psi}{\tilde{\alpha}^{n-2}} \big)\,,
$$
which we rewrite as 
\begin{align*}
	\big( \Lambda_{g_{0},q}-\Lambda_{g_{0},\tilde{q}}\big)\psi=&\big(B_{2}^{-1}-{\tilde{B_{2}}}^{-1}\big)\Lambda_{g}\big( \frac{\psi}{\alpha^{n-2}} \big)+{\tilde{B_{2}}}^{-1}\Lambda_{g}\big( \frac{\psi}{\alpha^{n-2}} - \frac{\psi}{\tilde{\alpha}^{n-2}}\big)+{\tilde{B_{2}}}^{-1}\big(\Lambda_{g}-\Lambda_{\tilde g}\big)\big( \frac{\psi}{\tilde{\alpha}^{n-2}}\big)\\
	&+{\tilde{B_{2}}}^{-1}{\tilde{B_{1}}}\big(  \frac{\psi}{\tilde{\alpha}^{n-2}}-\frac{\psi}{\alpha^{n-2}}\big)+\big({\tilde{B_{2}}}^{-1}{\tilde{B_{1}}}-B_{2}^{-1}B_{1}\big)\big( \frac{\psi}{\alpha^{n-2}} \big)\,,
\end{align*}
from which it follows using the definitions \eqref{defB1}, \eqref{defB2} of $B_1$ and $B_2$, \eqref{DNMapHyp-global} or \eqref{DNMapHyp-partial} and the boundary determination estimate \eqref{BDalpha} that there exist positive constants $C$ and $\theta_1$ such that 
\begin{align*}
	\lVert \big( \Lambda_{g_{0},q}-\Lambda_{g_{0},\tilde{q}}\big)\psi \rVert_{L^2(K)}\,\leq\,&C\,\lVert B_{2}^{-1}-{\tilde{B_{2}}}^{-1}\rVert_{L^\infty(K)}\cdot \lVert \frac{\psi}{\alpha^{n-2}}\rVert_{H^1(K)}+C\,\lVert \big( \frac{1}{\alpha^{n-2}} - \frac{1}{\tilde{\alpha}^{n-2}}\big)\psi \rVert_{H^1(K)}\\ &+C \epsilon\, \lVert \frac{\psi}{\tilde{\alpha}^{n-2}}\rVert_{H^1(K)}
	+\lVert {\tilde{B_{2}}}^{-1}{\tilde{B_{1}}}-B_{2}^{-1}B_{1}\rVert_{L^\infty(K)}\cdot \lVert \frac{\psi}{\alpha^{n-2}} \rVert_{L^{2}(K)}\,,
\end{align*}
and thus 
\begin{equation}\label{DNMapEst}
	\lVert \big( \Lambda_{g_{0},q}-\Lambda_{g_{0},\tilde{q}}\big)\psi \rVert_{L^2(K)}\,\leq  \,C\epsilon^{\theta_{1}}\lVert \psi \rVert_{H^1(K)}\,.
\end{equation} 
\end{proof}	

Let us now obtain some other consequences of these well-known boundary determination results. Recalling first that 
$$
  q=\frac{\Delta_{g_0}\alpha^{n-2}}{\alpha^{n-2}}\,, 
$$
a straightforward consequence of \eqref{BDalpha} is
\begin{equation} \label{qK}
  \| q(j,.) - \tilde{q}(j,.) \|_{L^\infty(K)} \leq C \epsilon^{\theta_2}, \quad j=0,1. 
\end{equation}
Second, using the separability condition \eqref{Separability}, we get :
$$
  \| \phi_1(j) - \tilde{\phi}_1(j) + \phi_2 - \tilde{\phi}_2 \|_{L^\infty(K)} \leq C \epsilon^{\theta_2}, \quad j=0,1. 
$$
But recall from Remark \ref{ConstantA}, 2. that we can choose $\phi_1(j) = \tilde{\phi}_1(j)$ for $j=0$ or $j=1$. Hence we obtain directly Hölder stability for the function $\phi_2$, \textit{i.e.}  :
\begin{equation} \label{Stability-phi2}
  \| \phi_2 - \tilde{\phi}_2 \|_{L^\infty(K)} \leq C \epsilon^{\theta_2}. 
\end{equation}
Finally, we obtain the following important corollary that will be used constantly in the next sections :
\begin{coro} \label{mu}
There exists a constant $C>0$ depending only on $K, g_K, c, B$ and a constant $0<\theta<\frac{1}{2}$ such that 
$$
  \forall k \geq 0, \quad | \mu_k^2 - \tilde{\mu}_k^2 | \leq C \epsilon^{\theta}. 
$$  
In particular, 
$$
   \forall k \geq 0, \quad | \mu_k - \tilde{\mu}_k | \leq C \frac{\epsilon^{\theta}}{\mu_k + \tilde{\mu}_k}.
$$
\end{coro}
\begin{proof}
Note first that we can write
$$
  -\Delta_{g_0} + \tilde{\phi}_2 =   -\Delta_{g_0} + \phi_2 + V, \quad V = \tilde{\phi}_2 - \phi_2,
$$
where, according to \eqref{Stability-phi2}
$$
  \| V \|_\infty \leq C \epsilon^{\theta_2}. 
$$
Thus, setting $\theta = \theta_2$, the result follows directly from the minmax characterization of the eigenvalues $\mu_k^2, \tilde{\mu}_k^2$ given for instance in \cite{Te}, Corollary 4.1.

\end{proof}


\Section{The radial stability estimates in the global Calder\'on problem} \label{GlobalStability}
 
\subsection{Preliminary results} 
 
In the next two sections, we will repeatedly use the Marchenko representations of the solutions $s_j, c_j, \ j=0,1$ of \eqref{radialz}. Following \cite{Mar2011}, p.9, we have for instance for all $\mu \in \R$:
\begin{equation} \label{s0}
  s_0(x,-\mu^2) = \frac{\sinh(\mu x)}{\mu} + \int_0^x H_0(x,t) \frac{\sinh(\mu t)}{\mu} dt,  
\end{equation} 
\begin{equation} \label{c0}
	c_0(x,-\mu^2) = \cosh(\mu x) + \int_0^x P_0(x,t) \cosh(\mu t) dt,  
\end{equation}  
and corresponding formulae for $s_1, c_1$. Moreover, using \eqref{ConstantA}, Gendron \cite{Gen2022}, p.30, has shown that there exists a constant $C_A > 0$ depending only on $A$ such that the kernels $H_0$ and $P_0$ satisfy 
\begin{equation} \label{EstHP} 
  \| H_0 \|_\infty + \| \partial_x H_0 \|_{\infty} + \| \partial_t H_0 \|_\infty \leq C_A, \quad \| P_0 \|_\infty + \| \partial_x P_0 \|_{\infty} + \| \partial_t P_0 \|_\infty \leq C_A.
\end{equation} 

We will exploit these formulae to prove several technical results. First, we introduce a notation. For a function $f = f(x,k)$ defined on $[0,1] \times \N$ and a positive sequence $(u_k)_{k \geq 0}$, we write 
$$
   f(x,k) = O_A(u_k) \quad \Longleftrightarrow \quad \exists C_A > 0 \, | \, \forall x \in [0,1], \forall k \geq 0, \ |f(x,k)| \leq C_A u_k.
$$

 \begin{lemma} \label{Delta}
 Using the above notation,  the asymptotics of the characteristic function $\Delta$ are given by :  
 $$
   \Delta(-\mu_k^2) = \frac{\sinh(\mu_k)}{\mu_k} + O_A \left(  \frac{e^{\mu_k}}{\mu_k^2}\right). 
 $$
 Even more precisely, we have 
 $$
  \Delta(-\mu_k^2) = \frac{\sinh(\mu_k)}{\mu_k} +  \bigg( \frac{1}{2}{\int_0^1 \phi_1(x) dx}\bigg) \frac{\cosh(\mu_k)}{\mu_k^2} +  O_A \left(  \frac{e^{\mu_k}}{\mu_k^3}\right). 
 $$  
 \end{lemma}
\begin{proof}
Using \eqref{s0} and integrating by parts, we obtain
\begin{align}\label{Asymptchar}
  s_0(x,-\mu_k^2) =& \, \frac{\sinh(\mu_k x)}{\mu_k}+\frac{H_0(x,1)\cosh(\mu_k)-H_0(x,0)}{\mu_k^2}\nonumber \\
  &-\frac{1}{\mu_k^2}\int_{0}^x \partial_t H_0(x,t)\cosh (\mu_k t)\,dt\,,
\end{align}
which implies using the estimates \eqref{EstHP} that 
$$
\Delta(-\mu_k^2)=\frac{\sinh(\mu_k)}{\mu_k}+O_A( \frac{e^{\mu_k}}{\mu_k ^2})\,.
$$
Now, recall from \cite{FY} (Theorem 1.2.3, p. 22) that 
\begin{equation}\label{intphi1}
H_0(x,x)=\frac{1}{2}\int_0^x\phi_1(t)dt\,,\quad H_0(x,0)=0\,,
\end{equation}
so that \begin{align*}
\Delta(-\mu_k^2)=\,  s_0(1,-{\mu_k}^2)=\, &\frac{\sinh(\mu_k )}{\mu_k}+\bigg(\frac{1}{2}\int_0^1\phi_1(t)dt\bigg)\,\frac{\cosh \mu_k}{\mu_k^2}\\
  &-\frac{1}{\mu_k^2}\int_{0}^1 \partial_t H_0(1,t)\cosh (\mu_k t)\,dt\,.
\end{align*}
Using \eqref{EstHP}, it follows that 
$$
\Delta(-\mu_k^2)=\frac{\sinh(\mu_k )}{\mu_k}++\bigg(\frac{1}{2}\int_0^1\phi_1(t)dt\bigg)\,\frac{\cosh \mu_k}{\mu_k^2}+O_A\bigg(\frac{\sinh(\mu_k)}{\mu_k^3}\bigg)\,,
$$
as claimed. 
\end{proof}

We proceed with a result on the asymptotics of the Weyl-Titchmarsh functions $M,N$. 

\begin{lemma} \label{AsympWT}
For all $\mu \in \R^+$, we have  
$$
   -M(-\mu^2), -N(-\mu^2)  = \mu + O_A \left( \frac{1}{\mu}  \right). 
$$  
\end{lemma}
\begin{proof}
Using \eqref{s0}, \eqref{c0}, \eqref{intphi1} and an integration by parts, we obtain
\begin{align*}
  s_0(x,-\mu^2) =& \,\frac{\sinh(\mu x)}{\mu} +\bigg(\frac{1}{2\mu^2}\int_0^x\phi_1(t)dt\bigg) \cosh \mu +O_A\big(\frac{e^\mu}{\mu^3}\big)\,,\\
c_0(x,-\mu^2) =& \,\cosh(\mu x) +\bigg(\frac{1}{2\mu}\int_0^x\phi_1(t)dt\bigg) \sinh \mu +O_A\big(\frac{e^\mu}{\mu^2}\big)\,,
\end{align*}  
from which it follows that 
$$
-M(-\mu^2)= \, \frac{c_0(1,-\mu^2)}{s_0(1,-\mu^2)}=\frac{\cosh \mu\bigg(1+\frac{1}{2\mu}\big(\int_0^1\phi_1(t)dt\big) \tanh \mu +O_A(\frac{1}{\mu^2}) \bigg)}{\frac{\sinh \mu}{\mu}\bigg(1+\frac{1}{2\mu}\big(\int_0^1\phi_1(t)dt \big) \coth \mu +O_A(\frac{1}{\mu^2})\bigg)}\,,
$$
which gives in turn 
\begin{align*}
=\mu\big(1+O\big(e^{-2\mu}\big)\big)&\bigg(1+\frac{1}{2\mu}\int_0^1\phi_1(t)dt(1+O(e^{-2\mu}))+O_A(\frac{1}{\mu^2})\bigg)\\
\times & \bigg(1-\frac{1}{2\mu}\int_0^1\phi_1(t)dt(1+O(e^{-2\mu}))+O_A(\frac{1}{\mu^2})\bigg)\,.
\end{align*}
Carrying out the multiplication in this last expression, we see that the terms involving $\frac{1}{2\mu}\int_0^1\phi_1(t)dt$ cancel out, leaving
$$
-M(-\mu^2)=\mu+O_A(\frac{1}{\mu})+O(\mu e^{-2\mu})= \mu + O_A(\frac{1}{\mu}),
$$
as claimed.
\end{proof}	

We end this subsection with a result on the monotony of the Weyl-Titchmarsh and characteristic functions. 

\begin{lemma}\label{Lemme0} 
 We have:
 \begin{enumerate}[i)] 
 \item The Weyl-Titchmarsh functions $\lambda\mapsto M(\lambda)$ and $\lambda\mapsto N(\lambda)$ are increasing for $\lambda \in \mathbb{R}^{-}$.
 \item $\exists\, \lambda_{0} \in \mathbb{R}^{-}$ depending only on $A$ and $(K,g_{K})$ such that the characteristic function $\lambda\mapsto \Delta(\lambda)$ is decreasing $\forall \, \lambda \in (-\infty, \lambda_{0})$.
 \end{enumerate}
 \end{lemma}
   \begin{proof} We first prove \emph{i)}. From \eqref{seqode} and \eqref{radialz}, we obtain, $\forall \, \lambda_1, \lambda_2 \in \mathbb{R}^{-}$,
   \[
   W(s_{0}(x,\lambda_1),s_{0}(x,\lambda_2))'=(\lambda_1-\lambda_2)s_{0}(x,\lambda_1)s_{0}(x,\lambda_2)\,,
   \]
   which implies upon integration
    \begin{equation}\label{Wronskid}
    s_{0}(1,\lambda_1)s_{0}'(1,\lambda_2)-s_{0}'(1,\lambda_1)s_{0}(1,\lambda_2)=(\lambda_1-\lambda_2)\int_{0}^{1}s_{0}(x,\lambda_1)s_{0}(x,\lambda_2)\,dx\,.
    \end{equation}
    Using Proposition 4.4 in \cite{Gen2022}, we have
    \[
        s_{0}(1,\lambda_1)s_{0}'(1,\lambda_2)-s_{0}'(1,\lambda_1)s_{0}(1,\lambda_2)=\Delta(\lambda_1)\Delta(\lambda_2)\big[-N(\lambda_2)+N(\lambda_1)\big]\,,
    \]
  so that \eqref{Wronskid} is equivalent to
      \[ 
      \frac{N(\lambda_1)-N(\lambda_2)}{\lambda_1-\lambda_2}=\frac{1}{\Delta(\lambda_1)\Delta(\lambda_2)}\int_{0}^{1}s_{0}(x,\lambda_1)s_{0}(x,\lambda_2)\,dx\,.
         \]
   Taking the limit $\lambda_1 \to \lambda_2 = \lambda$, we obtain
       \[ 
       \frac{dN}{d\lambda}(\lambda)=\frac{1}{\Delta^{2}(\lambda)}\int_{0}^{1}s_{0}^{2}(x,\lambda)\,dx \geq 0\,,\quad \forall\, \lambda \in \mathbb{R}^{-}\,.
         \]
    The argument is similar for $M(\lambda)$. 
     
    \noindent Next we proceed with the proof of \emph{ii)}. We start again from \eqref{seqode} and \eqref{radialz} to obtain, $\forall \, \lambda_1, \lambda_2 \in \mathbb{R}^{-}$, 
    \[
   W(s_{0}(x,\lambda_1),s_{1}(x,\lambda_2))'=(\lambda_1-\lambda_2)s_{0}(x,\lambda_1)s_{1}(x,\lambda_2)\,.
   \]
   We integrate from $0$ to $1$ and use again Proposition 4.4 in \cite{Gen2022} to obtain 
        \[      
        \frac{\Delta(\lambda_1)-\Delta(\lambda_2)}{\lambda_1-\lambda_2}=\int_{0}^{1}s_{0}(x,\lambda_1)s_{1}(x,\lambda_2)\,dx\,.
   \]      
      Taking the limit $\lambda_1 \to \lambda_2 = \lambda$, we obtain
     \begin{equation}\label{derivDelta}
       \frac{d\Delta}{d\lambda}(\lambda)=\int_{0}^{1}s_{0}(x,\lambda)s_{1}(x,\lambda)\,dx \,.
         \end{equation}
  The analysis of the sign of the right-hand side of \eqref{derivDelta} is a bit more involved than that of the corresponding term part \emph{i)} of the lemma. Recalling the Marchenko representations of the solutions $s_0, s_1$ of \eqref{radialz}, given for $\mu \in \mathbb{R}^+$ by 
  \begin{align*}
  s_0(x,-\mu^2)=& \frac{\sinh (\mu x)}{\mu}+\int_{0}^{x}H_0(x,t) \frac{\sinh (\mu t)}{\mu}\,dt\,,\\
  s_1(x,-\mu^2)=&-\frac{\sinh(\mu(1-x))}{\mu}+\int_{x}^{1}H_1(x,t) \frac{\sinh (\mu(1- t))}{\mu}\,dt\,,
  \end{align*}
  we can split the right-hand side of \eqref{derivDelta} into a sum of four terms as follows:
   \[ 
   \int_{0}^{1}s_{0}(x,\lambda)s_{1}(x,\lambda)\,dx = I_1+I_2+I_3+I_4\,,
    \]
    where 
      \begin{align*}
      I_1=& -\int_0^1\frac{\sinh(\mu x)\sinh(\mu(1-x))}{\mu^2}dx\,,\\
I_2=&-\int_0^1\int_0^x\frac{\sinh(\mu t)\sinh(\mu(1-x))H_0(x,t)}{\mu^2}\,dt\,dx\,,\\
I_3=&\int_0^1\int_x^1\frac{\sinh(\mu x)\sinh(\mu(1-t))H_1(x,t)}{\mu^2}\,dt\,dx\,,\\
I_4=&\int_0^1\int_0^x\int_x^1\frac{\sinh(\mu t)\sinh(\mu(1-s))H_0(x,t)H_1(x,s)}{\mu^2}\,ds\,dt\,dx\,.
        \end{align*}
        We now estimate each of these terms in turn. For $I_1$, the result is very easy to obtain; we have 
         \begin{equation}\label{estimI1}
           I_1= -\frac{\cosh \mu}{2\mu^2}+\frac{\sinh \mu}{2\mu^3}=-\frac{e^\mu}{2\mu^2}\big(1+O(\frac{1}{\mu})\big)\,.
           \end{equation}
        For the remaining terms, we shall use standard identities satisfied by the hyperbolic functions and the fact that there exists a positive constant $C_A$ such that $\left|H_{0,1}(x,t)\right| \leq C_A$ for $0\leq x,t \leq 1$. For $I_2$, this gives 
          \begin{align*}
          \left|I_2 \right|\leq& \,\frac{C_A}{2\mu^2}\,\left|\int_0^1\int_0^x \big[\cosh(\mu(t-x+1))+\cosh(\mu(t+x-1))\big]\,dt\,dx \,\right|\\
          \leq&\,\frac{C_A}{2\mu^3}\,\left| \int_0^1\big[\sinh \mu + \sinh (\mu(2x-1))\big]\,dt\,dx \,\right| \,,
         \end{align*}
         which gives 
         \begin{equation}\label{estimI2}
           \left|I_2 \right|\leq C_A \frac{e^\mu}{2\mu^3}\,.
         \end{equation}
         Likewise, a similar calculation for $I_3$ gives 
          \begin{equation}\label{estimI3}
            \left|I_3 \right|\leq C_A\frac{e^\mu}{\mu^3}\,.
              \end{equation}
Finally, we shall estimate $I_4$ by splitting this quantity into a sum of two terms which we will treat separately, We first write using the standard hyperbolic identities and the change of variables $v=t-s,\,u=t+s$, 
   \[ 
   I_4=I_{4,1}+I_{4,2}\,,
       \]
    where
      \begin{align*}
     I_{4,1}= &\frac{1}{2\mu^2} \int_0^1\int_0^x\int_{t-1}^{t-x}H_0(x,t)H_1(x,t-v)\cosh(\mu(v+1))dv\,dt\,dx\,,\\
     I_{4,2}= &-\frac{1}{2\mu^2} \int_0^1\int_0^x\int_{t+x}^{t+1}H_0(x,t)H_1(x,u-t)\cosh(\mu(u-1))du\,dt\,dx\,.
       \end{align*}
 By applying Fubini on $I_{4,1}$, we obtain
   \[
    I_{4,1}= \frac{1}{2\mu^2}\int_{-1}^{0}\cosh(\mu(v+1))\bigg(\int_{0}^1\int_{x+v}^{1+v}H_0(x,t)H_1(x,t-v)dt\,dx\bigg)\,dv\,,
    \]
    which implies 
       \[
       \left| I_{4,1} \right| \leq \frac{C_{A}}{2\mu^3}\sinh \mu\,.
    \]
    For $I_{4,2}$ we apply Fubini twice to obtain
     \begin{align*}
         I_{4,2}
         =&-\frac{1}{2\mu^2}\int_0^1\bigg(\int_0^u\int_{0}^{u-x}H_0(x,t)H_1(x,u-t)dt\,dx\bigg)\cosh(\mu(u-1))du\\
         &-\frac{1}{2\mu^2}\int_1^2\bigg(\int_{u-1}^1\int_{0}^{u-x}H_0(x,t)H_1(x,u-t)dt\,dx\,\bigg)\cosh(\mu(u-1))du\,,
         \end{align*}
         which implies in turn that 
           \[
       \left| I_{4,2} \right| \leq \frac{C_{A}}{2\mu^3}\sinh \mu\,.
    \]
    Combining the estimates we just obtained for $I_{4,1}$ and $I_{4,2}$, we conclude that 
      \begin{equation}\label{estimI4}
       \left| I_4 \right| \leq C_A \frac{e^\mu}{\mu^3}\,.
            \end{equation}
  Finally, putting together the estimates \eqref{estimI1},  \eqref{estimI2},  \eqref{estimI3},  \eqref{estimI4} we derived for $I_1,I_2,I_3,I_4$, we obtain
    \[
    \frac{d \Delta(-\mu^2)}{d(-\mu^2)}=-\frac{e^\mu}{2\mu^2}\big(1+O_{A}(\frac{1}{\mu})\big)\,,
      \]
      which implies that $\exists \mu^{*}(A)\,|\, \forall \mu \geq \mu^{*}$, we have
    \begin{equation}
    \frac{d \Delta(-\mu^2)}{d(-\mu^2)}\leq 0\,,
    \end{equation}
    which proves \emph{ii)}
  \end{proof}

 \begin{coro} \label{Increasing}
 We have:
  \begin{enumerate}[(i)] 
  \item The maps $\mu \in \mathbb{R}^{+}\to -M(-\mu^2)$ and $\mu \in \mathbb{R}^{+}\to -N(-\mu^2)$ are increasing.
  \item $\exists\, \mu^{*} \in \mathbb{R}$ depending only on $A$ and $(K,g_{K})$ such that the map $\mu \mapsto \frac{1}{\Delta(-\mu^{2})}$ is increasing  $\forall \, \mu \in [-\mu^{*}, \infty)$.

   \end{enumerate}
 \end{coro}
 \begin{proof}
 The claims \emph{i)} and \emph{ii)} follow immediately from \emph{i)} and \emph{ii)} in Lemma \ref{Lemme0} by the change of variables $\lambda = -\mu^2$, which implies that $\frac{d}{d\lambda}= -2\mu \frac{d}{d\mu}$.
 \end{proof}


\subsection{The fundamental estimates for the radial stability}

\begin{lemma}\label{Lemme1} 
	There exists a constant $C_A$ depending only on $A, K, g_K$ such that $\forall\, k\geq 0$  we have:
\begin{equation}\label{Lemme1statement}
 \left|M(-\mu_{k}^{2})-\tilde{M}(-\tilde{\mu}_{k}^{2})\right|<C_{A}\,\epsilon^{\frac{1}{2}}\,,\quad |N(-\mu_{k}^{2})-\tilde{N}(-\tilde{\mu}_{k}^{2})|<C_{A}\,\epsilon^{\frac{1}{2}}\,.
\end{equation}
   
  \end{lemma} 
 \begin{proof}
 We shall only give the proof of the statement for the Weyl-Titchmarsh function $M$, the argument being similar for $N$. We begin by considering the case for which the values of $\mu_k$ are large, where we have the asymptotics
 \[
 -M(-\mu^2)=\mu+O_{A}\left( \frac{1}{\mu} \right)\,,
 \]
 and similar asymptotics for $\tilde M$. We therefore obtain using Corollary \ref{mu}
  \[
  \left| M(-\mu_k^2)-{\tilde {M}}(-{\tilde {\mu}}_k^2)\right|=  \left| \mu_k-{\tilde {\mu}}_k +O_{A}\left(\frac{1}{\mu_k}+\frac{1}{{\tilde {\mu}}_k}\right) \right| \leq \frac{Ce^\theta}{\mu_k+{\tilde {\mu}}_k} + C_A \left(\frac{1}{\mu_k}+\frac{1}{{\tilde {\mu}}_k} \right)\leq \frac{C_{A}}{\mu_k}\,.
   \]
We now use the Weyl law  
\begin{equation} \label{WL}
  \mu_k = c_{n-1} k^{\frac{1}{n-1}} + O(1),
\end{equation}
from which we deduce that there exists a constant $c_0$ depending only on $K,g_K$ such that 
\begin{equation} \label{muk}
  \forall k \geq 0, \quad c_{n-1} k^{\frac{1}{n-1}} - c_0 \leq \mu_k \leq 	c_{n-1} k^{\frac{1}{n-1}} + c_0.  
\end{equation}	
We conclude that we can choose $k_{1}=k_{1}(\epsilon) \geq \left(  \frac{c_0 + \epsilon^{-a}}{c_{n-1}} \right)^{n-1}$ with $0<a<1$ such that 
$$
  \mu_k \geq \epsilon^{-a}, \quad \forall k\geq k_{1}(\epsilon)\,
$$
and thus
\begin{equation}\label{smallk}
     \left| M(-\mu_k^2)-{\tilde {M}}(-{\tilde {\mu}}_k^2)\right| \leq C_A \epsilon^{a}\, \quad \forall k \geq k_1(\e).
 \end{equation}
 Next we consider the case for which the values of $\mu_k$ are small, which we treat using a min-max argument. Indeed, recall that the sequence $ -M(-\mu_{k}^{2})$ of upper diagonal components of the matrix \eqref{DNqFouriermat} gives the eigenvalues of the partial DN map $\Lambda_{0}:=\Dngammaqzero$, which is self-adjoint on $L^2(K,dvol_{g_K})$. Since the function $\mu_k \mapsto -M(-\mu_k^2)$ is increasing according to Lemma \ref{Increasing}, we can therefore use the min-max characterization :
       \[
       -M(-\mu_{k}^{2})=\min_{\psi_1,\ldots,\psi_k}\, \max_{\psi \in \Span{\{\psi_1,\ldots ,\psi_k \}}}\frac{\big(\psi,\Lambda_0 \psi\big)_{L^2(K,dvol_{g_{K}})}}{\lVert \psi \rVert ^{2}_{L^2(K,dvol_{g_{K}})}}\,.
        \]
 From this it follows from our main hypothesis \eqref{DNMapHyp-global} that 
    \begin{align*}
      -M(-\mu_{k}^{2})\leq  &  \max_{\psi \in \Span{\{{\tilde{Y}}_{1}} \ldots {\tilde{Y}}_{k}\}}\frac{\big(\psi,\tilde{\Lambda}_0 \psi \big)_{L^2(K,dvol_{g_{K}})}+\big(\psi,(\Lambda_0 -\tilde{\Lambda}_0)\psi \big)_{L^2(K,dvol_{g_{K}})}}{\lVert \psi \rVert ^{2}_{L^2(K,dvol_{g_{K}})}} \\
      \leq &   -{\tilde {M}}(-{\tilde {\mu}}_k^2) +\epsilon  \max_{\psi \in \Span{\{{\tilde{Y}}_{1}} \ldots {\tilde{Y}}_{k}\}} \frac{\lVert \psi \rVert _{H^1(K,dvol_{g_{K}})}}{\lVert \psi \rVert _{L^2(K,dvol_{g_{K}})}}\\
        \leq &   -{\tilde {M}}(-{\tilde {\mu}}_k^2)+\epsilon(1+\mu_{k}^2)^{\frac{1}{2}}\,.
      \end{align*}
By symmetry, we obtain therefore
\begin{equation}\label{almostlargek}
        \left| M(-\mu_k^2)-{\tilde {M}}(-{\tilde {\mu}}_k^2)\right| \leq \epsilon (1+\mu_{k}^2)^{\frac{1}{2}}\,,\quad \forall k\geq 0\,.
\end{equation}
From \eqref{muk} and our choice of $k_1(\e)$, we now get from \eqref{almostlargek}
\begin{equation}\label{largek}
          \left| M(-\mu_k^2)-{\tilde {M}}(-{\tilde {\mu}}_k^2)\right| \leq  \epsilon^{1-a}\,,\quad \forall k\leq k_1(\epsilon). 
\end{equation}
Putting \eqref{smallk} and \eqref{largek} together, we obtain
\[
  \left| M(-\mu_k^2)-{\tilde {M}}(-{\tilde {\mu}}_k^2)\right| \leq C_A (\epsilon^{1-a}+\epsilon^a)\,,\quad \forall k\geq 0\,,
\]
for a constant $C_A$ depending only on $A, K, g_K$. This gives the claim \eqref{Lemme1statement} upon setting $a=1/2$.
\end{proof}

\begin{lemma}\label{Lemme2} 
   There exist constants $k_{0}=k_{0}(A)$, $0<\delta <\frac{1}{2}$ and $C_A > 0$ depending only on $A,K,g_K$ such that $\forall\, k\geq k_{0}$ we have:
\[
\left|\frac{1}{\Delta(-\mu_{k}^{2})}-\frac{1}{{\tilde\Delta}(-\tilde{\mu}_{k}^{2})}\right|<C_{A}\,\epsilon^{1-\delta}\,.
\]
\end{lemma}
\begin{proof}
  We begin by considering the case for which the values of $\mu_k$ are large, for which we have the asymptotics
 \[
 \Delta(-\mu^{2})=\frac{\sinh \mu}{\mu}\bigg(1+O_{A}(\frac{1}{\mu})\bigg)\,.
 \]
 This gives using Corollary \ref{mu}
     \begin{align*}
     \left| \frac{1}{\Delta(-\mu_{k}^{2})}-\frac{1}{{\tilde\Delta}(-\tilde{\mu}_{k}^{2})}\right| \leq &     \left| \frac{\mu_k}{\sinh \mu_k} -  \frac{{\tilde{\mu}}_k }{\sinh {\tilde{\mu}}_k}\right| + C_{A}\left| \frac{1 }{\sinh \mu} -  \frac{1 }{\sinh {\tilde{\mu}}_k}\right| \\
     \leq & \,\frac{\left|\mu_k - {\tilde{\mu}}_k \right|}{\sinh \mu_k}+ {\tilde{\mu}}_k \left| \frac{\sinh {\tilde{\mu}}_k-\sinh \mu_k}{\sinh \mu_k \,\sinh {\tilde{\mu}}_k}\right| + \frac{C_{A}}{\sinh \mu_k} \\
      \leq & \, \frac{C\,e^\theta}{\mu_k \sinh \mu_k}+2{\tilde{\mu}}_k \left| \frac{\cosh\big( ({\tilde{\mu}}_k+\mu_k)/2 \big)\sinh\big( ({\tilde{\mu}}_k-\mu_k)/2 \big)}{\sinh \mu_k \,\sinh {\tilde{\mu}}_k}\right| + \frac{C_{A}}{\sinh \mu_k} \\
\leq & \, \frac{C\,e^\theta}{\mu_k \sinh \mu_k} + \frac{C\,e^\theta}{ \sinh \mu_k} + \frac{C_{A}}{\sinh \mu_k}  \\
\leq & \, \frac{C_{A}}{e ^{\mu_k}}\,.
    \end{align*}
Using the Weyl law \eqref{WL} and the same argument as in the previous lemma, it follows that there exists $k_{2}=k_{2}(\epsilon) \geq \left( \frac{a\ln \epsilon + c_0}{c_{n-1}} \right)^{n-1}$ such that $\forall\, k\geq k_{2}$, we have, for $0<a<1$, 
    \begin{equation}\label{muklargeagain}
    \left|\frac{1}{\Delta(-\mu_{k}^{2})}-\frac{1}{{\tilde\Delta}(-\tilde{\mu}_{k}^{2})}\right| \leq C_A \epsilon^{a}\,.
     \end{equation}
 We now move on to the case for which the values of $\mu_k$ are small, which we treat again using a min-max argument. Indeed, recall that the sequence of off-diagonal components $ -1/\Delta(-\mu_{k}^{2})$ of the matrix \eqref{DNqFouriermat} gives the eigenvalues of the partial DN map $\Lambda_{0}:=\Dngammaqzerotoone$, which is self-adjoint on $L^2(K,dvol_{g_K})$. Moreover, the map $\mu_k \mapsto -\frac{1}{\Delta(-\mu_k^2)}$ is increasing for all $k \geq k_0$ according to Corollary \ref{Increasing}. Applying a min-max argument similar to the one used in the proof of Lemma \ref{Lemme1}, we obtain that for all $k \geq k_0=k_0(A)$ 
 \[
     \left|\frac{1}{\Delta(-\mu_{k}^{2})}-\frac{1}{{\tilde\Delta}(-\tilde{\mu}_{k}^{2})}\right|\leq  C \epsilon \,\mu_{k}\,.
 \]
 It follows then that for all $k$ such that $k_0\leq k \leq k_2$ (which is equivalent to having $\mu_k \leq -C_{A}\ln \epsilon$ for a constant $C_A$ depending only on $A, K, g_K$), we have 
  \begin{equation}\label{muksmallagain}
        \left|\frac{1}{\Delta(-\mu_{k}^{2})}-\frac{1}{{\tilde\Delta}(-\tilde{\mu}_{k}^{2})}\right|\leq C_{A}(- \epsilon \ln \epsilon)\,.   
 \end{equation}
 Putting together \eqref{muklargeagain} and \eqref{muksmallagain}, we obtain for all $k \geq k_0$ :
          \[
                \left|\frac{1}{\Delta(-\mu_{k}^{2})}-\frac{1}{{\tilde\Delta}(-\tilde{\mu}_{k}^{2})}\right|\leq C_{A}(- \epsilon \ln \epsilon+\epsilon^{a})\,,
            \]
            which gives our claim upon setting $a=1-\delta$.
     \end{proof}


\subsection{Reduction to a moment problem} \label{ReductionMoment}

We start from the identity

\begin{lemma} 
For all $k \geq 0$, we have :
\begin{multline*}
	\int_0^1 [ \tilde{\phi}_1(x) - \phi_1(x)  +\tilde{\mu}_k^2 - \mu_k^2 ] c_0(x,\mu_k) \tilde{s}_0(x,\tilde{\mu}_k) dx \\ = \tilde{\Delta}(-\tilde{\mu}_k^2) \left( \frac{1}{\Delta(-\mu_k^2)} - \frac{1}{\tilde{\Delta}(-\tilde{\mu}_k^2)} \right) + M(-\mu_k^2) \Delta(-\mu_k^2) \tilde{\Delta}(-\tilde{\mu}_k^2) \left( \tilde{N}(-\tilde{\mu}_k^2) - N(\mu_k^2) \right).  
\end{multline*}
\end{lemma}	
\begin{proof}
This is an easy adaptation of the calculus given in \cite{Gen2022}, Proposition 4.4 and Lemma 4.5. 
\end{proof}

This identity can be conveniently rewritten as
\begin{equation} \label{a0}
\begin{split}
	\int_0^1 [ \tilde{\phi}_1(x) - \phi_1(x) ] c_0(x,\mu_k) \tilde{s}_0(x,\mu_k) dx = & \quad \tilde{\Delta}(-\tilde{\mu}_k^2) \left( \frac{1}{\Delta(-\mu_k^2)} - \frac{1}{\tilde{\Delta}(-\tilde{\mu}_k^2)} \right) \\
	& + M(-\mu_k^2) \Delta(-\mu_k^2) \tilde{\Delta}(-\tilde{\mu}_k^2) \left( \tilde{N}(-\tilde{\mu}_k^2) - N(\mu_k^2) \right) \\	
    & + ( \mu_k^2 - \tilde{\mu}_k^2) \int_0^1 c_0(x,\mu_k) \tilde{s}_0(x, \mu_k) dx, \\
    & + \int_0^1 [ \tilde{\phi}_1(x) - \phi_1(x) ] c_0(x,\mu_k) (\tilde{s}_0(x,\mu_k) - \tilde{s}_0(x,\tilde{\mu}_k)) dx.
\end{split}
\end{equation}
Next, using the Marchenko representations \eqref{s0}-\eqref{c0} and following the work of Gendron \cite{Gen2022}, p. 32, there exists a bounded operator $R : L^2(0,1) \longrightarrow L^2(0,1)$ such that the l.h.s of \eqref{a0} can be rewritten as :
\begin{equation} \label{a1}
  \int_0^1 [ \tilde{\phi}_1(x) - \phi_1(x) ] c_0(x,\mu_k) \tilde{s}_0(x,\mu_k) dx = \frac{1}{\mu_k} \int_0^1 \sinh(2\mu_k x) R[\tilde{\phi}_1 - \phi_1](x) dx. 
\end{equation}
Now, multiplying \eqref{a1} by $2\mu_k e^{-2\mu_k}$ and following the calculations of Gendron \cite{Gen2022}, p.33, we obtain
\begin{align*} 
  I := & \quad 2\mu_k e^{-2\mu_k}  \int_0^1 [ \tilde{\phi}_1(x) - \phi_1(x) ] c_0(x,\mu_k) \tilde{s}_0(x,\mu_k) dx \\
  = & \quad \int_0^{+\infty} e^{-2\mu_k x}\big[R[\tilde {\phi}_{1}-\phi_1](1-x){\bf{1}}_{[0,1]}(x)-R[\tilde {\phi}_{1}-\phi_1](x-1){\bf{1}}_{[1,2]}(x)\big]dx,
\end{align*}
that is the l.h.s. multiplied by $2\mu_k e^{-2\mu_k}$ takes the form of a simple Laplace transform. At this stage, we have proved that for all $k \geq 0$ 
\begin{equation} \label{a2}
	\begin{split}
		I= & \quad 2\mu_k e^{-2\mu_k} \Bigg\{ 
		\tilde{\Delta}(-\tilde{\mu}_k^2) \left( \frac{1}{\Delta(-\mu_k^2)} - \frac{1}{\tilde{\Delta}(-\tilde{\mu}_k^2)} \right) \\
		& + M(-\mu_k^2) \Delta(-\mu_k^2) \tilde{\Delta}(-\tilde{\mu}_k^2) \left( \tilde{N}(-\tilde{\mu}_k^2) - N(\mu_k^2) \right) \\	
		& + ( \mu_k^2 - \tilde{\mu}_k^2) \int_0^1 c_0(x,\mu_k) \tilde{s}_0(x,\mu_k) dx \\
		& + \int_0^1 [ \tilde{\phi}_1(x) - \phi_1(x) ] c_0(x,\mu_k) (\tilde{s}_0(x,\mu_k) - \tilde{s}_0(x,\tilde{\mu}_k)) dx.
		\Bigg\}, \\
		= & \quad I_1 + I_2 + I_3 + I_4.
	\end{split}
\end{equation}

In order to estimate the terms $I_1 + I_2 + I_3$ in this expression, we recall using Lemmas \ref{Delta} and \ref{AsympWT} and the Marchenko representations \eqref{s0}-\eqref{c0} of the solutions $s_0,c_0$ that for all $k \geq 0$ and for all $x \in [0,1]$ :
\begin{align}
\left| M(-\mu_{k}^{2})\right|&\leq C_A \mu_{k}\,,\quad \left| \Delta(-\mu_{k}^{2})\right|\leq C_A \frac{e^{\mu_k}}{\mu_{k}}\,, \label{a3}\\
\left|\tilde{s}_0(x,-{\tilde \mu}_k^2)\right|&\leq C_A \frac{e^{\tilde{\mu}_k}}{\tilde{\mu}_{k}}\,,\quad \left|{c}_0(x,-{\mu}_k^2)\right|\leq C_A e^{{\mu}_k}\,. \label{a4}
\end{align}
Recalling next from Corollary \ref{mu} that 
$$
\left| \mu_k-\tilde{\mu}_k\right|\leq C\frac{\epsilon^\theta}{\mu_k} \leq C\frac{\epsilon^\theta}{\mu_1}, \quad \frac{\tilde{\mu}_k}{\mu_k} \leq 1 + \frac{C \epsilon^\theta}{\mu_k} \leq 1 + \frac{C \epsilon^\theta}{\mu_1}, \quad \forall k \geq 1, 
$$
we get :
\begin{align*}
	\left|I_1 + I_2 + I_3\right| \leq C_A\bigg\{  & \left|  {\tilde{N}}(-\tilde{\mu}_{k}^2)-N(-\mu_{k}^2) \right|  + e^{-\mu_k}  \left|  \frac{1}{\Delta(-\mu_{k}^{2})}-\frac{1}{{\tilde\Delta}(-\tilde{\mu}_{k}^{2})}    \right|  \\ 
	& + \left| \mu_k ^2- {\tilde \mu}_k ^2\right|  \bigg\}\,.
\end{align*}
At last, using Lemmas \ref{Lemme1} and \ref{Lemme2} and Corollary \ref{mu} again, we  conclude that there exists a constant $C_A > 0$ depending only on $A,K,g_K,c,B$ such that for all $k\geq k_0(A)$, we have 
$$
\left|I_1 + I_2 + I_3 \right| \leq C_A\bigg\{ \sqrt{\epsilon} + \epsilon^{1-\delta} + e^\theta \bigg\}\leq 2\ C_A e^\theta,
$$
for $\epsilon$ small enough.

In order to estimate the fourth term in the above identity, we need the following lemma
\begin{lemma}\label{Lemme4} 
	We have:
	\[
	{\tilde{s}}_{0}(x,-\tilde{\mu}_{k}^2)= {\tilde{s}}_{0}(x,-\mu_{k}^2)+O_{A}\big(\epsilon^{\theta} \, \frac{e^{\mu_{k}}}{\mu_{k}}\big)\,.
	\]
\end{lemma} 
\begin{proof} Recall that we have the Marchenko representation \eqref{s0} for ${\tilde {s}}_0(x,-\mu_{k}^2)$ given by 
	\begin{equation}\label{March}
		{\tilde {s}}_0(x,-\tilde{\mu}_{k}^2)= \frac{\sinh (\tilde{\mu}_k x)}{\tilde{\mu}_k}+\int_{0}^{x} \tilde{H}_0(x,t) \frac{\sinh (\tilde{\mu}_k t)}{\tilde{\mu}_k}\,dt\,,        
	\end{equation}
	where $\left| \tilde{H}_{0}(x,t)\right| \leq C_A$ for $0\leq x,t \leq 1$. Using Corollary \ref{mu}, we can write
	\[
	\tilde{\mu}_k=\mu_k + r_k\,,\quad \left| r_k \right| \leq C\epsilon^\theta\,.
	\]
	It follows using the addition formula for hyperbolic sines that we have for all $x\in [0,1]$ and all $k\geq 0$
	\[
	\frac{\sinh (\tilde{\mu}_k x)}{\tilde{\mu}_k}=\frac{\sinh ({\mu}_k x)}{{\mu}_k}+O\big(\epsilon^{\theta} \, \frac{e^{\mu_{k}}}{\mu_{k}}\big)\,.
	\]    
	This entails the result using \eqref{March} and the above uniform bound on $\tilde{H}_0(x,t)$.
\end{proof}
Using Remark \ref{ConstantA} and \eqref{a4}, we obtain as an immediate corollary 
\begin{coro}\label{CoroLemme4} We have:
	\[ I_4 = 2\mu_k e^{-2\mu_k}
	\int_{0}^{1}[{\tilde {\phi}}_{1}(x)-\phi_{1}(x)]s_{0}(x,-\mu_{k}^2)( {\tilde{s}}_{0}(x,-\tilde{\mu}_{k}^2)-{\tilde{s}}_{0}(x,-\mu_{k}^2))\,dx=O_{A}\big(\epsilon^{\theta} \,\big) \,.
	\]
\end{coro}
Putting everything together, we obtain for all $k\geq k_0(A)$ 
\begin{equation} \label{a5}
\left|I\right| \leq C_A e^\theta \,,
\end{equation}
which is precisely the moment problem stated in the Introduction.


\subsection{Solution of the moment problem and stability estimates for the radial equation} \label{MomentPb}

Let us set 
$$
  g(x) = R[\tilde {\phi}_{1}-\phi_1](1-x){\bf{1}}_{[0,1]}(x)-R[\tilde {\phi}_{1}-\phi_1](x-1){\bf{1}}_{[1,2]}(x). 
$$
The change of variable $t =  e^{-x}$ transforms \eqref{a5} into :
$$
	\left| \int_0^1 t^{2\mu_k - 1} g(-\ln t) dt \right| \leq C_A \epsilon^\theta, \quad \forall k \geq k_0.  
$$	
Let us now set $\gamma = 2\mu_{k_0} - 1$ and 
$$
  h(t) = t^\gamma g(-\ln t), \quad \lambda_k = 2\mu_{k+k_0} - 1 - \gamma, \forall k \geq 0.   
$$
Thus we obtain the classical Hausdorff moment problem : 
\begin{equation} \label{b1}
	\left| \int_0^1 t^{\lambda_k} h(-\ln t) dt \right| \leq C_A \epsilon^\theta, \quad \forall k \geq 0, 
\end{equation}
for a sequence of moments 
$$
  0 = \lambda_0 \leq \lambda_1\leq \dots \leq \lambda_k \to +\infty, 
$$
and a function $h \in L^2(0,1)$ with support in $[e^{-2},1]$. 

This situation was already faced in our previous papers \cite{DKN7, DKN8} for a model corresponding to a warped balls or in Gendron \cite{Gen2022} in the case of warped products with transversal manifolds $(K,g_K) = (\mathbb{S}^{n-1}, d\omega^2)$ being the round sphere. Combining for instance the results in \cite{Gen2022}, section 4.4.1, which is closer to our model and the results in \cite{DKN7}, section 4.5. in which general transversal manifolds are studied, we can show that there exist constants $C_A>0$ depending only $A,K,g_K,c,B$ and $0<\theta <1$ such that 
\begin{equation} \label{b2}
  \| h \|_{L^2(0,1)} \leq C_A \left( \ln \left( \frac{1}{\epsilon} \right) \right)^{-\theta}. 
\end{equation}
Since $h = h_1 + h_2$ where supp $h_1 = [e^{-1},1]$ and supp $h_2 = [e^{-2},e^{-1}]$, we deduce from \eqref{b2} (following the lines given in \cite{Gen2022}) that 
\begin{equation} \label{b3}
	\| R[\phi_1 - \tilde{\phi}_1] \|_{L^2(0,1)} \leq C_A \left( \ln \left( \frac{1}{\epsilon} \right) \right)^{-\theta}. 
\end{equation}
At last, it was shown in Gendron \cite{Gen2022}, Section 4.4.2., that the operator $R$ is invertible with an inverse bounded by $\| R^{-1} \| \leq C_A$. So we conclude from \eqref{b3} that 
\begin{equation} \label{b3}
	\| \phi_1 - \tilde{\phi}_1 \|_{L^2(0,1)} \leq C_A \left( \ln \left( \frac{1}{\epsilon} \right) \right)^{-\theta}. 
\end{equation}


\Section{Radial stability estimates for the partial DN map $\Dngammaone$} \label{PartialStability}

In this section, we are now only considering the partial stability problem for the case of Dirichlet data supported on $\Gamma_1$ and Neumann data measured on $\Gamma_1$, that is 
\begin{equation}\label{DNMapGamma1Hyp}
\lVert \Dngammaonetilde -\Dngammaone \rVert_{\mathcal{B}(H^{1/2}(K), H^{-1/2}(K))}\leq \epsilon\,.
\end{equation}
The case where the Dirichlet and Neumann data are measured on $\Gamma_0$ can be treated similarly. Recall that we are furthermore assuming in this section that for a given $\nu >0$
\begin{equation} \label{c00}
\left| \int_{0}^{1}[{\tilde {\phi}}_{1}(x)-\phi_{1}(x)]\,dx \right| \leq C\epsilon^{\nu}\,,
\end{equation}
which can be thought of as smallness constraint on the integral of the difference of the radial components of the conformal factors $\alpha$ and $\tilde{\alpha}$. 

Under the hypothesis \eqref{DNMapGamma1Hyp}, the Hölder stability results at the boundary of Theorem \ref{Main2} follow directly from the results given in Section \ref{BD}. In particular, we recall from \eqref{Stability-phi2} and Corollary \ref{mu} that there exists a constant $C>0$ depending only on $K,g_K,c,B$ and a constant $0<\theta<\frac{1}{2}$ such that 
\[
\lVert \phi_2-\tilde{\phi_2}\rVert _{L^{\infty}(K)}\leq C \epsilon ^{\theta}\,.
\]
and 
\[
   |\mu_{k}^{2}- \tilde{\mu}_{k}^{2}|\leq C\,\epsilon^{\theta}\,,\quad \forall\, k\geq 0\,.
\]

Let us now prove the logarithmic stability result within the manifold given in Theorem \ref{Main2}. We start with an identity proved in \cite{Gen2022}, Lemma 5.3. For all $k \geq 0$, we have :
\begin{align}\label{c1}
          \int_{0}^{1}[{\tilde {\phi}}_{1}(x)-\phi_{1}(x)]s_{0}(x,-\mu_{k}^2) {\tilde{s}}_{0}(x,-\tilde{\mu}_{k}^2)\,dx =&\big[{\tilde N}(-{\tilde{\mu}_k}^2)-{N}(-{\mu}_k^2)\big]\Delta(-\mu_k^2){\tilde{\Delta}(-\tilde{\mu}}_k^2) \nonumber \\
        +&(\mu_k^{2}- \tilde{\mu}_k^{2})\int_{0}^{1}s_{0}(x,-\mu_k^2)( {\tilde{s}}_{0}(x,-\tilde{\mu}_k^2)\,dx\,,
\end{align}
We rewrite this identity in the convenient form
\begin{align}\label{c2}
    & \int_{0}^{1}[{\tilde {\phi}}_{1}(x)-\phi_{1}(x)]s_{0}(x,-\mu_{k}^2) {\tilde{s}}_{0}(x,-\mu_{k}^2)\,dx \nonumber \\
    & \hspace{2cm}= \big[{\tilde N}(-{\tilde{\mu}}_k^2)-{N}(-{\mu}_k^2)\big]\Delta(-\mu_k^2){\tilde{\Delta}(-\tilde{\mu}}_k^2) \nonumber \\
	& \hspace{2cm} +(\mu_k^{2}- \tilde{\mu}_k^{2})\int_{0}^{1}s_{0}(x,-\mu_k^2)( {\tilde{s}}_{0}(x,-\tilde{\mu}_k^2)\,dx\,, \\
	& \hspace{2cm} +\int_{0}^{1}[{\tilde {\phi}}_{1}(x)-\phi_{1}(x)]s_{0}(x,-\mu_{k}^2) ({\tilde{s}}_{0}(x,-\mu_{k}^2) - {\tilde{s}}_{0}(x,-\tilde{\mu}_{k}^2))\,dx. \nonumber
\end{align}

Using Proposition 5.4 in \cite{Gen2022}, we know that there exists an operator $D:L^2(0,1)\to L^2(0,1)$ such that the l.h.s. of \eqref{c2} can be written as 
\begin{align}\label{c3}
    \forall k \geq 1, \quad \int_{0}^{1}[{\tilde {\phi}}_{1}(x)-\phi_{1}(x)]s_{0}(x,-\mu_k^2) {\tilde{s}}_{0}(x,-\mu_k^2)\,dx=&\frac{1}{\mu_k^2}  \int_{0}^{1}\cosh(2\mu_k x)D[{\tilde {\phi}}_{1}-\phi_{1}](x)\,dx \nonumber \\
    -&\frac{1}{\mu_k^2}   \int_{0}^{1}[{\tilde {\phi}}_{1}(x)-\phi_{1}(x)]\,dx \,. 
\end{align}
We thus get from \eqref{c2} and \eqref{c3} for all $k \geq 1$ 
\begin{align}\label{c4}
	\frac{1}{\mu_k^2}  \int_{0}^{1}\cosh(2\mu_k x)D[{\tilde {\phi}}_{1}-\phi_{1}](x)\,dx  =&\big[{\tilde N}(-{\tilde{\mu}}_k^2)-{N}(-{\mu}_k^2)\big]\Delta(-\mu_k^2){\tilde{\Delta}(-\tilde{\mu}}_k^2) \nonumber \\
	+&(\mu_k^{2}- \tilde{\mu}_k^{2})\int_{0}^{1}s_{0}(x,-\mu_k^2)( {\tilde{s}}_{0}(x,-\tilde{\mu}_k^2)\,dx\,, \\
	+&\int_{0}^{1}[{\tilde {\phi}}_{1}(x)-\phi_{1}(x)]s_{0}(x,-\mu_{k}^2) ({\tilde{s}}_{0}(x,-\mu_{k}^2) - {\tilde{s}}_{0}(x,-\tilde{\mu}_{k}^2))\,dx. \nonumber \\
	+& \frac{1}{\mu_k^2}   \int_{0}^{1}[{\tilde {\phi}}_{1}(x)-\phi_{1}(x)]\,dx \,. \nonumber
\end{align}
Up to multiplying this identity by $2\mu_k^2e^{-2\mu_k}$ and letting 
\[
R[{\tilde {\phi}}_{1}-\phi_{1}](x):=D[{\tilde {\phi}}_{1}-\phi_{1}](1-x){\bf{1}}_{[0,1]}(x)+D[{\tilde {\phi}}_{1}-\phi_{1}](x-1){\bf{1}}_{[1,2]}(x)\,,
\]
we may rewrite the left-hand side of \eqref{c4} as 
\[
\int_{0}^{\infty}e^{-2\mu_k x}R[{\tilde {\phi}}_{1}-\phi_{1}](x)\,dx \, . 
\]
Eventually, we have obtained the identity for all $k \geq 1$ : 
\begin{align}\label{c5}
	 \int_{0}^{\infty}e^{-2\mu_k x}R[{\tilde {\phi}}_{1}-\phi_{1}](x)\,dx = \ & \ 2\mu_k^2e^{-2\mu_k} \Bigg\{ \big[{\tilde N}(-{\tilde{\mu}}_k^2)-{N}(-{\mu}_k^2)\big]\Delta(-\mu_k^2){\tilde{\Delta}(-\tilde{\mu}}_k^2) \nonumber \\
	+&(\mu_k^{2}- \tilde{\mu}_k^{2})\int_{0}^{1}s_{0}(x,-\mu_k^2)( {\tilde{s}}_{0}(x,-\tilde{\mu}_k^2)\,dx\,, \nonumber \\
	+ & \int_{0}^{1}[{\tilde {\phi}}_{1}(x)-\phi_{1}(x)]s_{0}(x,-\mu_{k}^2) ({\tilde{s}}_{0}(x,-\mu_{k}^2) - {\tilde{s}}_{0}(x,-\tilde{\mu}_{k}^2))\,dx, \nonumber \\
	+& \frac{1}{\mu_k^2}   \int_{0}^{1}[{\tilde {\phi}}_{1}(x)-\phi_{1}(x)]\,dx \, .\Bigg\}.  \\
	= & \ I_1 + I_2 + I_3 + I_4. \nonumber 
\end{align}
As in Section \ref{ReductionMoment}, we use Lemmas \ref{Delta}, \ref{AsympWT} and \ref{Lemme1}, Corollaries \ref{mu} and \ref{CoroLemme4} and \eqref{a5} to estimate the terms $I_1 + I_2 + I_3$. Precisely, we show that there exists a constant $C_A >0$ depending only $A, K, g_K, c, B$ such that
\begin{equation} \label{c6}
  | I_1 + I_2 + I_3| \leq C_A (\sqrt{\epsilon} + \epsilon^\theta) \leq C_A \epsilon^\theta. 
\end{equation}
Now, using the hypothesis \eqref{c00} and \eqref{a4}, we easily estimate the term $I_4$ by 
\begin{equation} \label{c7}
  |I_4 | \leq C_A \epsilon^{\nu}\,.
\end{equation}
Putting \eqref{c6} and \eqref{c7} together, we obtain 
\begin{equation} \label{c8}
  | I | = \left| \int_{0}^{\infty}e^{-2\mu_{k}x}R[{\tilde {\phi}}_{1}-\phi_{1}](x)\,dx\right| \leq C_{A} \epsilon^{\theta}\,, \quad \forall k \geq 1.  
\end{equation}

We thus face once again the same kind of moment problems as in Section \ref{MomentPb}. Using the same strategy used in \cite{DKN7, DKN8, Gen2022} and using that $R$ is invertible with a bounded inverse satisfying $\| R^{-1} \| \leq C_A$, we can show from \eqref{c8} that 
\begin{equation} \label{c9}
	\| \phi_1 - \tilde{\phi}_1 \|_{L^2(0,1)} \leq C_A \left( \ln \left( \frac{1}{\epsilon} \right) \right)^{-\theta}, 
\end{equation}
for a constant $C_A>0$ depending only on $A,K,g_K,c,B$ and a constant $0<\theta<1$.


\section{Stability estimates for the conformal factors}

Our goal in this section is to translate the stability estimates we have obtained on the difference $q-{\tilde{q}}$ of the potentials in terms of the difference $\alpha-{\tilde{\alpha}}$ of the potentials. 

Recall first that 
   \begin{align*}
q-{\tilde{q}}&=\alpha^{-(n-2)}\Delta_{g_{0}}\alpha^{n-2}-{\tilde{\alpha}}^{-(n-2)}\Delta_{g_{0}}{\tilde{\alpha}}^{n-2}\\
&=\phi_1(x)-\tilde{\phi_1}(x)+\phi_2(\omega)-\tilde{\phi_2}(\omega)\,.
\end{align*}
From \eqref{b3} or \eqref{c9}, we see that
\begin{equation} \label{d1}
  \| q - \tilde{q} \|_{L^2(M)} \leq C_A  \left( \ln \left( \frac{1}{\epsilon} \right) \right)^{-\theta}.  
\end{equation}
Following the strategy used in \cite{CaSa2014}, p. 20-21, we need a simple calculation
\begin{lemma}\label{Lemmep6}
We have
      \begin{align*}
\alpha^{n-2}{\tilde{\alpha}}^{n-2}(q-{\tilde{q}})&={\tilde{\alpha}}^{n-2}\Delta_{g_{0}}\alpha^{n-2}-\alpha^{n-2}\Delta_{g_{0}}{\tilde{\alpha}}^{n-2}\\
&=\frac{1}{\sqrt{\left| g_{0}\right|}}\partial_{x^{j}}\big[\alpha^{n-2}{\tilde{\alpha}}^{n-2}g_{0}^{jk} \sqrt{\left| g_{0}\right|}\partial_{x^{k}}(\ln \alpha^{n-2}-\ln {\tilde{\alpha}}^{n-2})\big]\,.
\end{align*}
\end{lemma}

\begin{proof} We have
   \begin{align*}   
          &\frac{1}{\sqrt{\left| g_{0}\right|}}\partial_{x^{j}}\big[\alpha^{n-2}{\tilde{\alpha}}^{n-2}g_{0}^{jk} \sqrt{\left| g_{0}\right|}\partial_{x^{k}}(\ln \alpha^{n-2}-\ln {\tilde{\alpha}}^{n-2})\big]\\
         &=\frac{1}{\sqrt{\left| g_{0}\right|}}\partial_{x^{j}}\big({\tilde{\alpha}}^{n-2}g_{0}^{jk} \sqrt{\left| g_{0}\right|} \partial_{x^{k}}\alpha^{n-2} \big)-\frac{1}{\sqrt{\left| g_{0}\right|}}\partial_{x^{j}}\big({\alpha}^{n-2}g_{0}^{jk} \sqrt{\left| g_{0}\right|} \partial_{x^{k}}\tilde{\alpha}^{n-2} \big)\\
            &={\tilde{\alpha}}^{n-2}\Delta_{g_{0}}\alpha^{n-2}-\alpha^{n-2}\Delta_{g_{0}}{\tilde{\alpha}}^{n-2}-(\partial_{x^{j}}{\tilde{\alpha}}^{n-2})g_{0}^{jk}(\partial_{x^{k}}\alpha^{n-2})+(\partial_{x^{j}}\alpha^{n-2})g_{0}^{jk}(\partial_{x^{k}}\tilde{\alpha}^{n-2})\\
             &={\tilde{\alpha}}^{n-2}\Delta_{g_{0}}\alpha^{n-2}-\alpha^{n-2}\Delta_{g_{0}}{\tilde{\alpha}}^{n-2}\,.
        \end{align*}
          
\end{proof}
We shall use the following convenient consequence of this result:
\begin{lemma}\label{Lemmep7} 
	Let 
              \[
              g_{1}=(\alpha \tilde{\alpha})g_{0}\,,\quad v:=(n-2)\ln\frac{\alpha}{\tilde{\alpha}}\,.
              \]
          We have
                 \begin{equation}\label{elliptg1}
             \Delta_{g_{1}}v=\frac{q-\tilde{q}}{(\alpha{\tilde {\alpha}})^2}\,.  
                    \end{equation}
             
\end{lemma}
\begin{proof} Letting
         \[
         g_1= \beta g_0\,.
             \]
             we have 
                   \[
                        (g_1)^{-1}= \frac{1}{\beta} (g_0)^{-1}\,,\quad \left| g_{1}\right|=\beta^{n}\left| g_{0}\right|\,,
                  \]
                  and therefore
               \[
                 \Delta_{g_{1}}=\frac{1}{\beta^{\frac{n}{2}}\sqrt{\left| g_{0}\right|}}\partial_{x^{j}}\big(\beta^{\frac{n}{2}-1}g_{0}^{jk} \sqrt{\left| g_{0}\right|}\partial_{x^{k}}\big)\,.
                   \]
                   The conclusion follows by letting $\beta = (\alpha \tilde{\alpha})^2$ in the preceding equation and applying Lemma \ref{Lemmep6}.
                  
\end{proof}


We first exploit Lemma \ref{Lemmep7} to obtain the following stability estimate on the difference of the conformal factors $\alpha$ and ${\tilde{\alpha}}$ for the \emph{global} Calder\'on problem. Our proof follows closely the argument used in \cite{CaSa2014}.

\begin{prop}
	There exist positive constants $C$ and $\theta'$ such that 
	\[
	\lVert \alpha -{\tilde{\alpha}}\rVert_{C^{0,r}(M)} \leq Ce^{\theta'}\,.
	\]
\end{prop}

\begin{proof}
We first apply the standard elliptic existence and regularity theorems to the pde \eqref{elliptg1} to obtain
\[
\lVert \ln \alpha - \ln{\tilde{\alpha}}\rVert_{H^{1}(M)}\leq C\bigg( \lVert \frac{q-\tilde{q}}{(\alpha{\tilde {\alpha}})^2} \rVert _{H^{-1}(M)}+\lVert \ln \alpha - \ln{\tilde{\alpha}} \rVert_{H^{1/2}(\partial M)}\bigg)\, ,
\]

\vspace{0.5cm}\noindent
for some positive constant $C$, from which we deduce that
\begin{equation}\label{Ellreg}
	\lVert \ln \alpha - \ln{\tilde{\alpha}}\rVert_{H^1(M)}\leq C \big( \lVert q-\tilde{q} \rVert _{L^{2}(M)} + \lVert \ln \alpha - \ln{\tilde{\alpha}} \rVert_{H^{1}(\partial M)}\big)\, .
\end{equation}      
Using \eqref{d1}, Theorem \ref{Main1}, \eqref{Angular}, it is immediate the r.h.s. of \eqref{Ellreg} is bounded by
\begin{equation} \label{d2}
	\lVert q-\tilde{q} \rVert _{L^{2}(M)} +  \lVert \ln \alpha - \ln{\tilde{\alpha}} \rVert_{H^{1}(\partial M)} \leq C \left( \ln \left( \frac{1}{\epsilon} \right) \right)^{-\theta}. 
\end{equation}
We finish the proof using a trick due to Caro, Garcia and Ruiz \cite{CGR2013}. Observe first that
\begin{equation} \label{d3}
	\forall f \in L^\infty(M), \quad |f(x)|^{\frac{n}{1-r}} \leq \| f \|_{L^\infty(M)}^{\frac{n}{1-r} -2} |f(x)|^2, \quad 0<r<1. 
\end{equation}
Thus we get from \eqref{Ellreg}-\eqref{d3} and Hypothesis \ref{Hyp} that
\begin{equation} \label{d4}
	\lVert \ln \alpha - \ln{\tilde{\alpha}}\rVert_{W^{1,\frac{n}{1-r}}(M)} \leq C 
	\lVert \ln \alpha - \ln{\tilde{\alpha}}\rVert_{H^1(M)}^{\frac{(1-r)2}{n}}
	\leq C \left( \ln \left( \frac{1}{\epsilon} \right) \right)^{-\theta'}. 
\end{equation}  
At last, we use Morrey embeddings to prove that \eqref{d4} implies 
\begin{equation}  \label{d5}
	\lVert \ln \alpha - \ln{\tilde{\alpha}}\rVert_{ C^{0,r}(M)} \leq \left( \ln \left( \frac{1}{\epsilon} \right) \right)^{-\theta'},  
\end{equation}        
which entails our main result \eqref{Radial}. 
\end{proof}


Let us finally exploit Lemma~\ref{Lemmep7} to derive a stability estimate for the difference of the conformal factors \(\alpha\) and \(\tilde{\alpha}\) for the \emph{partial} Calder\'on problem where the measures are made on $\Gamma_1$. We are able to obtain stability estimates in an interior region away from (\emph{but arbitrarily close to}) the boundary $\Gamma_0$.

For any \(\tau\in(0,1)\), we set
\[
M_\tau := [\tau,1]\times K.
\]
As a consequence of standard Carleman estimates (recalled in Appendix \ref{B}), we first obtain a localized logarithmic stability estimate on $M\tau$.

\begin{prop}[Localized logarithmic stability]
	\label{prop:Holder-local}
	Fix \(r\in(0,1)\). Then, for any \(\tau\in(0,1)\), there exists a constant
	\(C_\tau>0\) such that
	\begin{equation}\label{eq:Holder-prop}
		\|\alpha - \tilde{\alpha}\|_{C^{0,r}(M_\tau)}
		\;\le\;
		C_\tau\,|\ln(1/\varepsilon)|^{-\theta_\tau},
	\end{equation}
	where the exponent \(\theta_\tau>0\) depends explicitly on the distance
	\(\tau\) to the boundary $\Gamma_0$ and is given by
	\[
	\theta_\tau
	=
	\frac{2(1-r)}{n}\,
	\frac{\theta\,\tau}{2-\tau}.
	\]
	Moreover, \(C_\tau\) can be chosen so that \(C_\tau \lesssim \tau^{-2}\).
\end{prop}

\begin{proof}
Combining Lemma~\ref{Lemmep7} with Theorem~\ref{thm:bourgeois-cylinder} in Appendix~\ref{B}, we obtain
\begin{equation} \label{eq:bourgeois}
\begin{aligned}
				\| \ln \alpha - \ln \tilde{\alpha} \|_{H^{1}(M_\tau)}
				\;\le\;
				C_\tau\Big(&
				e^{(1-\tau)/h}
				\Big(
				\big\| \tfrac{q-\tilde{q}}{(\alpha \tilde{\alpha})^{2}} \big\|_{L^{2}(M)}
				\\
				&\quad
				+ \|\ln \alpha - \ln \tilde{\alpha}\|_{H^{1}(\Gamma_1)}
				+ \|\partial_\nu (\ln \alpha - \ln \tilde{\alpha})\|_{L^{2}(\Gamma_1)}
				\Big)
				\\
				&\quad
				+ e^{-\tau/(2h)}\,
				\| \ln \alpha - \ln \tilde{\alpha} \|_{H^{1}(M)}
				\Big), \quad \forall \,\, 0<h<h_0.
			\end{aligned}
		\end{equation}
		Using~\eqref{d1}, Theorem~\ref{Main2} (estimate~\eqref{Angular-partial}), and recalling that the
		conformal factors \(\alpha\) and \(\tilde{\alpha}\) satisfies Hypothesis \ref{Hyp}, it follows that
		\begin{equation}
			\label{eq:bourgeois-eps}
			\| \ln \alpha - \ln \tilde{\alpha} \|_{H^{1}(M_\tau)}
			\;\le\;
			C_\tau\Big(
			e^{(1-\tau)/h}\,
			\bigl|\ln \tfrac{1}{\varepsilon}\bigr|^{-\theta}
			+ e^{-\tau/(2h)}
			\Big), \quad \forall 0<h<h_0..
		\end{equation}
		Here \(C_\tau>0\) depends only on \(\tau\) and on the geometry of \(M\); moreover,
		as shown in Appendix~\ref{B}, one can choose \(C_\tau\lesssim \tau^{-2}\). Optimizing with respect to the semiclassical parameter \(h\) is achieved by balancing
		the two terms in~\eqref{eq:bourgeois-eps}, which leads to the choice
		\[
		h = \frac{1-\tau/2}{\theta\,\ln|\ln(1/\varepsilon)|}.
		\]
		With this choice, we obtain the logarithmic stability estimate
	\begin{equation}
		\label{eq:H1-log-stability}
		\| \ln \alpha - \ln \tilde{\alpha} \|_{H^{1}(M_\tau)}
		\;\le\;
		C_\tau\,|\ln(1/\varepsilon)|^{-\kappa},
		\qquad
		\kappa=\frac{\theta\tau}{2-\tau}.
	\end{equation}
	We conclude the proof by adapting the device introduced by
	Caro, Garc\'{\i}a and Ruiz~\cite{CGR2013} to the present localized setting keeping track of the dependence on $\tau$ of our Hölder exponent. Let \(0<r<1\) and set
	\[
	p := \frac{n}{1-r} > n .
	\]
	We recall that for any \(f \in L^\infty(M_\tau)\),
	\begin{equation}\label{eq:CGR-pointwise}
		|f(x)|^{p}
		\;\le\;
		\|f\|_{L^\infty(M_\tau)}^{p-2}\,|f(x)|^{2},
		\qquad x\in M_\tau .
	\end{equation}
The pointwise bound
\eqref{eq:CGR-pointwise} implies, after integration over \(M_\tau\),
\[
\|f\|_{L^p(M_\tau)}
\;\le\;
\|f\|_{L^\infty(M_\tau)}^{1-\frac{2}{p}}\,
\|f\|_{L^2(M_\tau)}^{\frac{2}{p}}.
\]
Applying this inequality to \(f=\nabla(\ln\alpha-\ln\tilde{\alpha})\), and using the fact that the admissibility of the conformal factors (Hypothesis \ref{Hyp}) provides a uniform
\(L^\infty\) bound on the gradient, we obtain
\[
\|\nabla(\ln\alpha-\ln\tilde{\alpha})\|_{L^p(M_\tau)}
\;\le\;
C\,
\|\nabla(\ln\alpha-\ln\tilde{\alpha})\|_{L^2(M_\tau)}^{\frac{2}{p}}.
\]
Finally, combining this estimate with the analogous bound for
\(\ln\alpha-\ln\tilde{\alpha}\) itself yields
\begin{equation}\label{eq:Wp-estimate}
	\|\ln \alpha - \ln \tilde{\alpha}\|_{W^{1,p}(M_\tau)}
	\;\le\;
	C_\tau\,
	\|\ln \alpha - \ln \tilde{\alpha}\|_{H^{1}(M_\tau)}^{\frac{2}{p}},
\end{equation}
Combining~\eqref{eq:Wp-estimate} with the logarithmic stability estimate
\eqref{eq:H1-log-stability}, we infer that
\begin{equation}\label{eq:Wp-log}
	\|\ln \alpha - \ln \tilde{\alpha}\|_{W^{1,p}(M_\tau)}
	\;\le\;
	C_\tau\,|\ln(1/\varepsilon)|^{-\frac{2\kappa}{p}}.
\end{equation}
Finally, since \(p>n\), Morrey's embedding yields
	\[
	W^{1,p}(M_\tau)\hookrightarrow C^{0,r}(M_\tau).
	\]
	As a consequence, we obtain the localized H\"older stability estimate
	\begin{equation}\label{eq:Holder-final}
		\|\ln \alpha - \ln \tilde{\alpha}\|_{C^{0,r}(M_\tau)}
		\;\le\;
		C_\tau\,|\ln(1/\varepsilon)|^{-\theta_\tau},
	\end{equation}
	where the exponent \(\theta_\tau>0\) depends explicitly on the distance
	\(\tau\) to the boundary $\Gamma_0$ and is given by
	\[
	\theta_\tau
	=
	\frac{2\kappa}{p}
	=
	\frac{2(1-r)}{n}\,
	\frac{\theta\,\tau}{2-\tau}.
	\]
Whence the result stated in the proposition.	
\end{proof}

\medskip
\noindent
As an immediate consequence of Proposition~\ref{prop:Holder-local}, we obtain the following corollary.
It shows that, by letting the localization parameter \(\tau=\tau(\varepsilon)\) tend to \(0\) sufficiently slowly,
one still obtains a stability rate which is stronger than the global log--log estimates available for
Calder\'on-type inverse problems with partial data (see, e.g., Caro--Dos Santos Ferreira--Ruiz~\cite{CDR2016}).

\begin{coro}[Improved stability away from the forbidden boundary]
	\label{cor:stretched-loglog}
	Let \(r\in(0,1)\) be fixed and set
	\[
	\tau(\varepsilon)
	=
	\frac{1}{\sqrt{\ln \, \ln(1/\varepsilon)}}.
	\]
	Then, as \(\varepsilon\to0\),
	\[
	\| \alpha - \tilde{\alpha}\|_{C^{0,r}(M_{\tau(\varepsilon)})}
	\;\lesssim\;
	\ln \, \ln(1/\varepsilon)\,
	\exp\,\Bigl(
	-\frac{(1-r)\theta}{n}\,
	\sqrt{\ln \, \ln(1/\varepsilon)}
	\Bigr).
	\]
\end{coro}

\begin{rem} \label{rem:global-position}
	
	The stability estimate on $M_{\tau(\e)}$ obtained in Corollary~\ref{cor:stretched-loglog}
	exhibits an intermediate behavior between logarithmic and log--log stability.
	On the one hand, it is strictly stronger than any log--log type estimate:
	for every \(\kappa>0\),
	\[
	\ln \, \ln(1/\varepsilon)\,
	\exp\!\Bigl(
	-\frac{(1-r)\theta}{n}\,
	\sqrt{\ln \, \ln(1/\varepsilon)}
	\Bigr)
	=
	o\,\bigl((\ln \, \ln(1/\varepsilon))^{-\kappa}\bigr),
	\qquad \varepsilon\to0.
	\]
	On the other hand, it remains weaker than a genuine logarithmic stability,
	since for any \(\alpha>0\),
	\[
	\ln \, \ln(1/\varepsilon)\,
	\exp\!\Bigl(
	-\frac{(1-r)\theta}{n}\,
	\sqrt{\ln \, \ln(1/\varepsilon)}
	\Bigr)
	\gg
	(\ln(1/\varepsilon))^{-\alpha},
	\qquad \varepsilon\to0.
	\]
	
	\medskip
	\noindent
	More generally, Proposition~\ref{prop:Holder-local} allows for a flexible
	choice of the localization parameter \(\tau=\tau(\varepsilon)\to0\).
	For instance, for any integer \(k\ge3\), taking
	\[
	\tau(\varepsilon)=\frac{1}{\ln_k(1/\varepsilon)},
	\qquad
	\ln_1=\ln,\quad \ln_{j+1}=\ln\circ\ln_j,
	\]
	yields the estimate
	\[
	\|\ln \alpha-\ln\tilde\alpha\|_{C^{0,r}(M_{\tau(\varepsilon)})}
	\;\lesssim\;
	\bigl(\ln_k(1/\varepsilon)\bigr)^2\,
	\exp\!\Bigl(
	-\,\frac{(1-r)\theta}{n}\,
	\frac{\ln \, \ln(1/\varepsilon)}{\ln_k(1/\varepsilon)}
	\Bigr),
	\qquad \varepsilon\to0.
	\]
	Such bounds are strictly stronger than the  rate
	of Corollary~\ref{cor:stretched-loglog}, while still remaining weaker than
	any fixed negative power of \(\ln(1/\varepsilon)\).
	Moreover, there is no optimal choice of \(\tau(\varepsilon)\) under the
	constraint \(\tau(\varepsilon)\to0\): by letting \(\tau(\varepsilon)\) tend
	to zero arbitrarily slowly, one can obtain stability rates arbitrarily close
	to logarithmic, at the price of localizing farther away from the forbidden
	boundary \(\Gamma_0\).
\end{rem}

         
\appendix
\section{On the existence of Painlev\'e-Liouville Riemannian manifolds} \label{A}

A question that is crucial to the existence of Painlev\'e-Liouville Riemannian manifolds is that of the existence of positive solutions to an appropriately chosen boundary-value problem for the elliptic pde~\eqref{pdealpha} governing the conformal factor $\alpha$. This is addressed in the following:
\begin{prop} Suppose that $\phi_1+\phi_2 \in C^{k-2}(M)$ with $k \geq 2$, $\phi_1+\phi_2>0$ on $M$, $\eta \in H^{k-\frac{1}{2}}(\partial M)$ and $\eta>0$ on $\partial M$. Then the boundary-value problem
\begin{equation}\label{bvpalpha}
  \left\{ \begin{split}
      -\Delta_{g_0}w+(\phi_1+\phi_2)w=0  & \quad \text{on $M$}\\
      w=\eta & \quad \text {on $\partial M$}
       \end{split}  \right.     
\end{equation}
has a solution $w\in H^{k}(M)$. In particular if $k=2+\frac{n}{2}$, then $w\in C^{2}(M)$ and if $k=2+r+\frac{n}{2}$, then $w\in C^{2,r}(M)$.
\end{prop}
\begin{proof}
Let $\overline{\eta}\in H^{k}(M)$ be such that $\overline{\eta}|_{\partial M}=\eta$ and let $v=w-\overline{\eta}$. Then, letting 
$$
f= \Delta_{g_0}\overline{\eta}-(\phi_1+\phi_2)\overline{\eta}\,,
$$
The boundary-value problem~\eqref{bvpalpha} becomes 
\begin{equation}\label{bvpv}
  \left\{ \begin{split}
      -\Delta_{g_0}v+(\phi_1+\phi_2)v=f  & \quad \text{on $M$}\\
      v=0 & \quad \text {on $\partial M$}
       \end{split}  \right.     
\end{equation}

Our positivity hypothesis $\phi_1+\phi_2>0$ implies that $0\notin \sigma( -\Delta_{g_0}+(\phi_1+\phi_2))$, so that \eqref{bvpv} admits a unique solution given by 
$$
v= (-\Delta_{g_0}v+(\phi_1+\phi_2))^{-1}f\,,
$$
where $ (-\Delta_{g_0}v+(\phi_1+\phi_2))^{-1}:H^{s}(M)\to H^{s+2}(M)$. Our hypothesis $\phi_1+\phi_2 \in C^{k-2}(M)$ then implies that $f\in H^{k-2}(M)$, which combined with $\eta \in H^{k-\frac{1}{2}}(\partial M)$ gives that $v\in H^{k}(M)$ and therefore that $w\in H^k(M)$. Using the Sobolev embeddings
$$
H^k(M) \subset C^{p,r}(M)\,, \quad p+r = k-\frac{n}{2}\,,
$$
we see that we obtain a solution $w\in C^{2}(M)$ if $k=2+\frac{n}{2}$ and $w\in C^{2,r}(M)$ if $k=2+ r +\frac{n}{2}$. Finally, using the strong maximum principle for elliptic operators on compact manifolds with boundary (\cite{Aubin}, \&~8), we know that $w=\alpha^{n-2}$ is strictly positive on $M$. 
\end{proof}

\section{Conditional stability estimates} \label{B}

In this section, we prove a conditional stability estimate on the cylinder
\( M = [0,1]\times K \), endowed with a smooth Riemannian metric
\[
g_1 = c^4 \, g_0,
\]
where \( g_0 =dx^2 + g_K\)  and $c$ is a positive conformal factor. 
The manifold \((M,g_1)\) is a \emph{conformally transversally anisotropic} (CTA) manifold
in the sense of Dos Santos Ferreira, Kenig, Salo and Uhlmann~\cite{DKSU2009}.
In particular, CTA manifolds admit limiting Carleman weights,
which play a crucial role in the derivation of Carleman estimates.
In the present setting, a convenient choice of limiting Carleman weight is
\[
\varphi(x)=x,
\]
where \( x\in[0,1] \) denotes the radial coordinate on the cylinder \(M=[0,1]\times K\). We denote
\[
\Gamma_0 := \{0\}\times K,
\qquad
\Gamma_1 := \{1\}\times K,
\]
and, for any fixed \(\tau\in(0,1)\),
\[
M_\tau := [\tau,1]\times K.
\]
We introduce the elliptic operator
\[
P := -\Delta_{g_1}.
\]

\begin{thm}[Conditional stability on the cylinder]
	\label{thm:bourgeois-cylinder}
	Let \(\tau\in(0,1)\) be fixed. Then there exist constants \(C_\tau>0\) and
	\(h_0>0\), depending only on \(\tau\) and on the geometry of \(M\), such that
	for all \(0<h<h_0\) and for all \(u\in H^{2}(M)\),
	\begin{align}
		\label{eq:bourgeois-cylinder}
		\|u\|_{H^{1}(M_\tau)}
		\le
		C_\tau\Big(
		e^{(1-\tau)/h}
		\big(
		\|Pu\|_{L^{2}(M)}
		+
		\|u\|_{H^{1}(\Gamma_1)}
		+
		\|\partial_\nu u\|_{L^{2}(\Gamma_1)}
		\big)
		+
		e^{-\tau/(2h)}\,\|u\|_{H^{1}(M)}
		\Big).
	\end{align}
	Here \(\nu\) denotes the outward unit normal vector field on \(\partial M\)
	with respect to the metric \(g_1\).
	Moreover, the
	constant \(C_\tau\) can be chosen such that \(C_\tau \lesssim \tau^{-2}\).
\end{thm}

\begin{proof}
The proof is relatively classical. It follows the strategy described in
Isakov~\cite[p.~54]{Isakov2006} and in Theorem~2.3 of
Bourgeois~\cite{Bourgeois2010}, adapted to the cylindrical geometry, and
relies on the boundary Carleman estimates of
Kenig-Salo~\cite[Proposition~4.2]{KS2013} and
Li-L\"u~\cite[Corollary~1.2]{LiLu2026}.

\medskip
\noindent\textbf{Step~1: Reduction to homogeneous Cauchy data on \(\Gamma_1\).}
We first reduce the problem to the case of homogeneous Cauchy data on \(\Gamma_1\).

\vspace{0.2cm}\noindent
By standard elliptic theory on the cylinder, one can perform this reduction
by means of a lifting argument for elliptic boundary value problems; see,
for instance, Lions--Magenes~\cite{LionsMagenes1972}. 
More precisely, there exists a function \(u^{*}\in H^2(M)\) depending only on the boundary data
 such that
\[
u^{*}\big|_{\Gamma_1}=u\big|_{\Gamma_1},
\qquad
\partial_\nu u^{*}\big|_{\Gamma_1}
=
\partial_\nu u\big|_{\Gamma_1},
\]
and
\begin{equation}\label{eq:lift-H2}
	\|u^{*}\|_{H^2(M)}
	\le
	C\Big(
	\|u\|_{H^{1}(\Gamma_1)}
	+
	\|\partial_\nu u\|_{L^{2}(\Gamma_1)}
	\Big).
\end{equation}

\medskip\noindent
Hence, replacing \(u\) by \(u-u^{*}\), we may assume without loss of generality
that \(u\) and its normal derivative \(\partial_\nu u\) vanish on \(\Gamma_1\).

\medskip
\noindent\textbf{Step~2: Cutoff near the inaccessible boundary \(\Gamma_0\).}
We choose a cutoff function
\(\chi\in C^\infty([0,1])\) such that
\[
\chi \equiv 1 \quad \text{on } [\tfrac{\tau}{2},\,1],
\qquad
\chi \equiv 0 \quad \text{in a neighborhood of } x=0,
\qquad
\operatorname{supp}(\chi')\subset [\tfrac{\tau}{4},\,\tfrac{\tau}{2}].
\]
We then set
\[
v := \chi\,u,
\qquad
U := e^{\varphi/h}\,v,
\]
where \(\varphi(x)=x\). In particular, since \(\chi\equiv 1\) on \(M_\tau\), we have \(v=u\) on \(M_\tau\).
Moreover, by construction, \(U \in H^2(M)\cap H_0^1(M)\).

\medskip
\noindent\textbf{Step~3: Carleman estimates.}
We first recall the boundary Carleman estimate of
Kenig-Salo~\cite[Proposition~4.2]{KS2013} and
Li-L\"u~\cite[Corollary~1.2]{LiLu2026}.
Let \((M,g_1)\) be an admissible manifold and let
\(\varphi\) be a limiting Carleman weight.
Then there exist constants \(C>0\) and \(h_0>0\) such that, for all
\(0<h<h_0\) , $\delta >0$ and all \(U\in H^2(M)\cap H_0^1(M)\), one has
\begin{align}
	\label{eq:lilu-carleman-q0}
	& h^2 \left( \|U\|_{L^2(M)}^2
	+ \|h\nabla_{g_1} U\|_{L^2(M)}^2\right)
	\nonumber\\
	&\qquad
	+ \delta h^3 \|\partial_\nu U\|_{L^2(\{\partial_\nu\varphi\le -\delta\})}^2
	+ h^4 \|\partial_\nu U\|_{L^2(\{-\delta<\partial_\nu\varphi<h/3\})}^2
	\nonumber\\
	&\le
	C\Big(
	\|e^{\varphi/h}(h^2 P)(e^{-\varphi/h}U)\|_{L^2(M)}^2
	+ h^3 \|\partial_\nu U\|_{L^2(\{\partial_\nu\varphi\ge h/3\})}^2
	\Big).
\end{align}

\medskip\noindent
We now apply \eqref{eq:lilu-carleman-q0} to the function
\(U=e^{\varphi/h}v\) constructed above, with the linear weight
\(\varphi(x)=x\).
Since \(g_1=c^{4}g_0\) with \(c>0\), we have
\[
\partial_\nu \varphi = -\,c^{-2} \quad \text{on } \Gamma_0,
\qquad
\partial_\nu \varphi = \phantom{-}\,c^{-2} \quad \text{on } \Gamma_1,
\]
where \(\nu\) denotes the outward unit normal vector field on \(\partial M\)
with respect to the metric \(g_1\). As a consequence, for sufficiently small \(h\), we obtain the Carleman estimate

\begin{align}
	\label{eq:carleman-cylinder}
	h^2\|U\|_{H^1_{\mathrm{scl}}(M)}^2
	\;\le\;
	C\Big(
	\|e^{\varphi/h}(h^2 P)(e^{-\varphi/h}U)\|_{L^2(M)}^2
	+
	h^3 \|\partial_\nu U\|_{L^2(\Gamma_1)}^2
	\Big),
\end{align}
where
\(
\|U\|_{H^1_{\mathrm{scl}}(M)}^2
:= \|U\|_{L^2(M)}^2 + \|h\nabla U\|_{L^2(M)}^2.
\)

\medskip\noindent
In our situation, \(u\) and \(\partial_\nu u\) vanish on \(\Gamma_1\), and
the cutoff is chosen so that \(v=\chi u = u\) in a neighborhood of \(\Gamma_1\).
Hence \(U=e^{\varphi/h}v\) also satisfies \(\partial_\nu U|_{\Gamma_1}=0\), so that the
boundary term on \(\Gamma_1\) in \eqref{eq:carleman-cylinder} vanishes.
Dividing \eqref{eq:carleman-cylinder} by \(h^2\), taking square roots, and using that
\(\chi\equiv 1\) on \(M_\tau\), we arrive at
\begin{equation}
	\label{eq:carleman-cylinder-divided}
	\|e^{\varphi/h}u\|_{H^1_{\mathrm{scl}}(M_\tau)}
	\le
	\|e^{\varphi/h}\chi u\|_{H^1_{\mathrm{scl}}(M)}
	\le
	C\,h\,
	\|e^{\varphi/h}P(\chi u)\|_{L^2(M)}.
\end{equation}

\medskip
\noindent\textbf{Step~4: Weight comparison on \(M_\tau\).}
We start by lower bounding the left-hand side of \eqref{eq:carleman-cylinder-divided}.
Since \(\varphi(x)=x\) and \(M_\tau=[\tau,1]\times K\), we have \(\varphi\ge \tau\) on \(M_\tau\),
hence \(e^{\varphi/h}\ge e^{\tau/h}\) on \(M_\tau\). Moreover,
\[
h\nabla_{g_1}\!\bigl(e^{\varphi/h}u\bigr)
=
e^{\varphi/h}\bigl(h\nabla_{g_1} u+(\nabla_{g_1}\varphi)\,u\bigr).
\]
Thus, we obtain
\begin{equation}
	\label{eq:lhs-lowerbound-H1}
	\|e^{\varphi/h}u\|_{H^1_{\mathrm{scl}}(M_\tau)}
	\;\ge\;
	c\,e^{\tau/h}\,\|u\|_{H^1_{\mathrm{scl}}(M_\tau)}
	\;\ge\;
	c\,h\,e^{\tau/h}\,\|u\|_{H^1(M_\tau)} ,
\end{equation}
for some constant \(c>0\) depending only on
\(\|\nabla_{g_1}\varphi\|_{L^\infty(M)}\).

\medskip
\noindent\textbf{Step~5: Commutator decomposition.}
We now estimate the right-hand side of \eqref{eq:carleman-cylinder-divided}.
Using the identity
\[
P(\chi u)=\chi\,Pu+[P,\chi]u,
\]
we obtain, by the triangle inequality,
\begin{align}
	\label{eq:triangular-commutator}
	\|e^{\varphi/h}P(\chi u)\|_{L^2(M)}
	&\le
	\|e^{\varphi/h}\chi\,Pu\|_{L^2(M)}
	+
	\|e^{\varphi/h}[P,\chi]u\|_{L^2(M)} .
\end{align}
Since \(\chi\le 1\), we have
\begin{equation}
	\label{eq:Pu-weighted}
	\|e^{\varphi/h}\chi\,Pu\|_{L^2(M)}
	\le
	\|e^{\varphi/h}Pu\|_{L^2(M)}
	\le
	e^{1/h}\,\|Pu\|_{L^2(M)} .
\end{equation}
Moreover, \([P,\chi]\) is a first--order differential operator supported in
\[
S_\tau:=\operatorname{supp}(\chi')\times K\subset [\tfrac{\tau}{4},\,\tfrac{\tau}{2}]\times K.
\]
Since $\varphi(x)=x\le \tfrac{\tau}{2}$ on $S_\tau$, we have
\begin{equation}\label{eq:commutator-weighted}
	\|e^{\varphi/h}[P,\chi]u\|_{L^2(M)}
	\le e^{\tau/(2h)}\|[P,\chi]u\|_{L^2(S_\tau)} .
\end{equation}
Moreover, there exists a constant $C_\tau>0$ (depending only on $P$ and on the cutoff $\chi$) such that
\begin{equation}\label{eq:commutator-local}
	\|[P,\chi_\tau]u\|_{L^2(S_\tau)}\le C_\tau\,\|u\|_{H^1(S_\tau)} .
\end{equation}
For a standard cutoff varying on a length scale $\tau$, one has $C_\tau\lesssim \tau^{-2}$. Combining \eqref{eq:commutator-weighted}--\eqref{eq:commutator-local} yields
\[
\|e^{\varphi/h}[P,\chi]u\|_{L^2(M)}
\le C_\tau\,e^{\tau/(2h)}\,\|u\|_{H^1(S_\tau)} .
\]
Using the above estimates and \eqref{eq:Pu-weighted}, together with
\eqref{eq:triangular-commutator}, we obtain
\begin{equation}
	\label{eq:right-hand-bound}
	\|e^{\varphi/h}P(\chi u)\|_{L^2(M)}
	\lesssim
	e^{1/h}\,\|Pu\|_{L^2(M)}
	+
	C_\tau\,e^{\tau/(2h)}\,\|u\|_{H^1(S_\tau)}.
\end{equation}

\medskip
\noindent\textbf{Step~6: Conclusion.}
Combining the lower bound \eqref{eq:lhs-lowerbound-H1} with the upper bound
\eqref{eq:right-hand-bound}, we obtain
\[
h\,e^{\tau/h}\,\|u\|_{H^1(M_\tau)}
\;\le\;
C_\tau\,h \left( 
e^{1/h}\,\|Pu\|_{L^2(M)}
+
e^{\tau/(2h)}\,\|u\|_{H^1(S_\tau)} \right) .
\]
Dividing both sides by \(h\,e^{\tau/h}\), this yields
\begin{equation}
	\label{eq:final-stability}
	\|u\|_{H^1(M_\tau)}
	\le
	C_\tau\!\left(
	e^{(1-\tau)/h}\,\|Pu\|_{L^2(M)}
	+
	e^{-\tau/(2h)}\,\|u\|_{H^1(S_\tau)}
	\right).
\end{equation}
This concludes the proof.
\end{proof}



\begin{thebibliography}{99}


\bibitem{Aubin} Aubin, T. \emph{Non-linear analysis on manifolds. Monge-Amp\`ere equations}, Springer 1982, 204 pp.
\bibitem{Ale1988} Alessandrini, G., \emph{Stable determination of conductivity by boundary measurements}, Appl. Anal. 27 (1988), no. 1-3, 153-172.  
\bibitem{AlGa2001} Alessandrini, G., Gaburro, R., \emph{Determining conductivity with special anisotropy by boundary measurements}, SIAM J. Math. Anal. 33 (2001), no. 1, 153-171.
\bibitem{AlGa2009} Alessandrini, G., Gaburro, R., \emph{The local Calder\'on problem and the determination at the boundary of the conductivity}, Comm. Partial Differential Equations 34 (2009), no. 7-9, 918-936. 
\bibitem{Ale1997}  Alessandrini, G.,  \emph{Examples of instability in inverse boundary-value problems}, Inverse Problems 13 (1997), no. 4, 887-897.
\bibitem{Ale2007} Alessandrini, G., \emph{Open issues of stability for the inverse conductivity problem}, J. Inverse Ill-Posed Probl. 15 (2007), no. 5, 451-460. 

\bibitem{Bourgeois2010}
Bourgeois, L.,
\emph{About stability and regularization of ill-posed elliptic Cauchy problems: the case of \(C^{1,1}\) domains},
ESAIM, Math. Model. Numer. Anal. \textbf{44} (2010), no.~4, 715--735.

\bibitem{CDR2016} Caro, P., Dos Santos Ferreira, D., Ruiz, A., \emph{Stability estimates for the Calderón problem with partial data}, J. Differential Equations 260 (2016), no. 3, 2457-2489. 
\bibitem{CGR2013} Caro, P., García, A., Reyes, J., \emph{Stability of the Calderón problem for less regular conductivities}, J. Differential Equations 254 (2013), no. 2, 469-492.
\bibitem{CaSa2014} Caro, P., Salo, M., \emph{Stability of the Calder\'on problem in admissible geometries}, Inverse Probl. Imaging 8 (2014), no. 4, 939-957.
\bibitem{DKN2} Daud\'e T., Kamran N., Nicoleau F., \emph{Non-uniqueness results in the anisotropic Calderon problem with data measured on disjoints sets}, Ann. Inst. Fourier (Grenoble) 69 (2019), no. 1, 119-170.
\bibitem{DKN3} Daud\'e T., Kamran N., Nicoleau F., \emph{On the hidden mechanism behind non-uniqueness for the anisotropic Calder\'on problem with data on disjoint sets}, (2017), 
Ann. Henri Poincar\'e 20 (2019), no. 3, 859-887
\bibitem{DKN4} Daud\'e T., Kamran N., Nicoleau F., \emph{A survey of non-uniqueness results for the anisotropic Calder\'on problem with disjoint data}, 
Harvard CMSA Series in Mathematics, Volume 2: Nonlinear Analysis in Geometry and Applied Mathematics, ed. T. Collins and S.-T. Yau, (2018).
\bibitem{DKN5} Daud\'e T., Kamran N., Nicoleau F.,  \emph {Francois Separability and symmetry operators for Painlev\'e metrics and their conformal deformations}, SIGMA Symmetry Integrability Geom. Methods Appl. 15 (2019), Paper No. 069, 42 pp.
\bibitem{DKN6} Daud\'e T., Kamran N., Nicoleau F.,  \emph{ On nonuniqueness for the anisotropic Calder\'on problem with partial data}, Forum Math. Sigma 8 (2020), Paper No. e7, 17 pp.
\bibitem{DKN7} Daud\'e T., Kamran N., Nicoleau F.,  \emph{Stability in the inverse Steklov problem on warped product Riemannian manifolds}, J. Geom. Anal. 31 (2021), no. 2, 1821-1854.
\bibitem{DKN8} Daud\'e T., Kamran N., Nicoleau F.,  \emph{Local H\"older stability in the inverse Steklov and Calder\'on problems for radial Schr\"odinger operators and quantified resonances}, Ann. Henri Poincar\'e 25 (2024), no. 8, 3805-3830. 
\bibitem{DKN2026} Daud\'e T., Kamran N., Nicoleau F.,  \emph{Stability in the inverse anisotropic Calder\'on problem on warped product manifolds}, in preparation (2026).  
\bibitem{DKSU2009} Dos Santos Ferreira, D., Kenig, C., Salo, Mi., Uhlmann, G., \emph{Limiting Carleman weights and anisotropic inverse problems}, Invent. Math. 178 (2009), no. 1, 119-171.
\bibitem{DKLS2016} Dos Santos Ferreira, D., Kurylev, Y., Lassas, M., Salo, M., \emph{The Calderón problem in transversally anisotropic geometries}, J. Eur. Math. Soc. (JEMS) 18 (2016), no. 11, 2579-2626.
\bibitem{FY} Freiling G., Yurko V., \emph{Inverse Sturm-Liouville problems and their applications},  Nova Science Publishers, Inc., Huntington, NY, 2001. x+356 pp
\bibitem{Gen2022} Gendron, G. \emph{Stability estimates for an inverse Steklov problem in a class of hollow spheres}, Asymptot. Anal. 126 (2022), no. 3-4, 323-377. 
\bibitem{Gui2017} Guillarmou, C., \emph{Lens rigidity for manifolds with hyperbolic trapped sets}, J. Amer. Math. Soc. 30 (2017), no. 2, 561-599.
\bibitem{HeWa2016} Heck, H., Wang, J.-N., \emph{Optimal stability estimate of the inverse boundary value problem by partial measurements}, Rend. Istit. Mat. Univ. Trieste 48 (2016), 369-383. 

\bibitem{Isakov2006}
Isakov V., \emph{Inverse Problems for Partial Differential Equations},
Applied Mathematical Sciences, $\mathbf{127}$,
Springer, New York (2006).

\bibitem{KY2002} Kang, H., Yun, K., \emph{Boundary determination of conductivities and Riemannian metrics via local Dirichlet-to-Neumann operator}, SIAM J. Math. Anal. 34 (2002), no. 3, 719-735. 

\bibitem{KS2013} Kenig C., Salo M, \emph{The Calderon problem with partial data on manifolds and applications}, Analysis \& PDE $\mathbf{6}$, no. 8, (2013), 2003-2048.

\bibitem{KRS2021} Koch, H., R\"uland, A.,Salo, M., \emph{On instability mechanisms for inverse problems}, Ars Inven. Anal., 2021, no. 7, 93, 2769-8505.

\bibitem{LiLu2026}
Li, Z., L\"u, Q.,
\emph{Carleman estimates for second-order elliptic operators with limiting weights: An elementary approach},
Sci. China Math., 69 (2026), 251--268.

\bibitem{LionsMagenes1972}
Lions, J.-L., Magenes, E.,
\emph{Non-homogeneous boundary value problems and applications. Vol.~I},
Springer-Verlag, Berlin--Heidelberg--New York, 1972.



 
\bibitem{Man2001} Mandache, N., \emph{Exponential instability in an inverse problem for the Schrödinger equation}, Inverse Problems 17 (2001), no. 5, 1435-1444.
\bibitem{Mar2011} Marchenko, V. A. \emph{Sturm-Liouville operators and applications}, Revised edition. AMS Chelsea Publishing, Providence, RI, 2011. xiv+396 pp. 
\bibitem{Nov2011} Novikov, R. G., \emph{New global stability estimates for the Gel'fand-Calder\'on inverse problem}, Inverse Problems 27 (2011), no. 1, 015001, 21 pp


\bibitem{PSUZ2019} Paternain, G., Salo, M. Uhlmann, G., Zhou, H., \emph{The geodesic X-ray transform with matrix weights}, Amer. J. Math. 141 (2019), no. 6, 1707-1750.
\bibitem{Sha1994} Sharafutdinov, V. A., \emph{Integral Geometry of Tensor Fields}, VSP, Utrecht (1994).
\bibitem{Te} Teschl G., \emph{Mathematical Methods in Quantum Mechanics}, Graduate Studies in Mathematics Vol. 99, AMS Providence, Rhode Island, (2009).
\bibitem{U1} Uhlmann G., \emph{Electrical impedance tomography and Calderon's problem}, Inverse Problems $\mathbf{25}$, (2009), 123011, 39p.
\bibitem{U2} Uhlmann G., \emph{Inverse problems: seeing the unseen}, Bull. Math. Sci. $\mathbf{4}$ (2014), no. 2, 209-279. 
\bibitem{UhVa2016} Uhlmann, G., Vasy, A., \emph{The inverse problem for the local geodesic ray transform}, Invent. Math.  (2016), 205, 83-120.

\end{thebibliography}

\noindent \footnotesize{Universit\'e Marie et Louis Pasteur, CNRS, LmB (UMR 6623), F-25000, Besan{\c{c}}on, France. \\
CNRS - Universit\'e de Montr\'eal CRM - CNRS \\	
\emph{Email adress}: thierry.daude@univ-fcomte.fr \\

\noindent Department of Mathematics and Statistics, McGill University,\\ \small  Montreal, QC, H3A 0B9, Canada.  \\
\emph{Email adress}: niky.kamran@mcgill.ca \\

\noindent Laboratoire de Math\'ematiques Jean Leray, UMR CNRS 6629, \\ \small 2 Rue de la Houssini\`ere BP 92208, F-44322 Nantes Cedex 03 \\
\emph{Email adress}: francois.nicoleau@univ-nantes.fr}

\end{document}